\RequirePackage{fix-cm}
\documentclass[smallextended]{svjour3}       
\smartqed  
\usepackage{amsmath,amssymb}
\usepackage{latexsym}            
\usepackage{fixltx2e}
\usepackage{ulem}
\usepackage{epsfig,subfigure}
\usepackage{epstopdf}
\usepackage{amsfonts,amscd}
\usepackage[ruled,vlined]{algorithm2e}
\usepackage{color,comment}
\usepackage{amsbsy,bm}

\usepackage{scalerel,stackengine}
\newcommand\reallywidehat[1]{%
	\savestack{\tmpbox}{\stretchto{%
			\scaleto{%
				\scalerel*[\widthof{\ensuremath{#1}}]{\kern-.6pt\bigwedge\kern-.6pt}%
				{\rule[-\textheight/2]{1ex}{\textheight}}
			}{\textheight}%
		}{0.5ex}}%
	\ensurestackMath{\stackon[1pt]{#1}{\tmpbox}}%
}

\newcommand\sA{{\mathcal A}}
\newcommand\sB{{\mathcal B}}

\newcommand\sE{{\mathcal E}}

\newcommand\sN{{\mathcal N}}

\newcommand\sT{{\mathcal T}}

\newcommand\sV{{\mathcal V}}
\newcommand\sW{{\mathcal W}}

\newcommand\Bf{\mathbf{f}}
\newcommand\bK{\mathbf{K}}
\newcommand\bn{\mathbf{n}}
\newcommand\bq{\mathbf{q}}
\newcommand\br{\mathbf{r}}
\newcommand\bx{\mathbf{x}}
\newcommand\by{\mathbf{y}}
\newcommand\bv{\mathbf{v}}
\newcommand\bw{\mathbf{w}}
\newcommand\bsw{\boldsymbol{w}}

\newcommand\bF{\mathbf{F}}

\newcommand\bP{\mathbf{P}}
\newcommand\bS{\mathbf{S}}

\newcommand\bW{\mathbf{W}}

\newcommand\intT{\displaystyle{\int_{0}^{T}}}
\newcommand\sumK{\displaystyle{\sum_{\bK_h\in\G_h}}}
\newcommand\intK{\displaystyle{\int_{\bK_h}}}
\newcommand\intpK{\displaystyle{\int_{\partial \bK_h}}}
\newcommand\sume{\displaystyle{\sum_{e \in \mathcal{E}}}}

\def \be{\begin{equation}}
\def \ee{\end{equation}}
\def \beaa{\begin{eqnarray}}
\def \eeaa{\end{eqnarray}}
\def \bes{\begin{eqnarray}}
\def \ees{\end{eqnarray}}
\def \bns{\begin{eqnarray*}}
	\def \ens{\end{eqnarray*}}

\newcommand{\ki}{{\mbox{\raise.2ex\hbox{$\chi$}}\hspace{.2ex}}}

\newcommand{\G}{{\Gamma}}

\renewcommand{\tilde}{\widetilde}

\renewcommand{\bar}{\overline}
\renewcommand{\hat}{\widehat}

\newcommand{\at}{\m\text{ at }}

\newcommand{\m}{\hspace{1em}}

\newtheorem{thm}{Theorem}

\newtheorem{lma}[thm]{Lemma}
\newtheorem{defn}{Definition}
\newtheorem{rmk}{Remark}
\begin{document}

\title{   Local Discontinuous Galerkin Methods for Solving Convection-Diffusion and Cahn-Hilliard Equations on Surfaces
}


\author{Shixin Xu   \and  Zhiliang Xu  
}


\institute{Shixin Xu \at
	Duke Kunshan University,
              8 Duke Ave, Kunshan Shi, Suzhou Shi, Jiangsu Sheng, China, 215316\\
           \and
           Zhiliang Xu \at
              Department of Applied and Computational
              Mathematics and Statistics, University of Notre Dame, Notre Dame, IN
              46556\\
               \email{zxu2@nd.edu}\\
}

\date{Received: date / Accepted: date}

\maketitle

\begin{abstract}
Local discontinuous Galerkin methods are developed for solving \textcolor{blue}{second order and fourth order} time-dependent partial differential equations defined on \textcolor{blue}{static 2D manifolds}. These schemes are  second-order accurate with surfaces triangulized by planar triangles and careful design of numerical fluxes. The schemes are proven to be energy stable.
Various numerical experiments are provided to validate the new schemes.
\end{abstract}

\section{Introduction}

Numerically solving partial differential equations (PDEs) defined on manifolds has drawn  a lot of attention recently. This is because models of many applications, such as surface diffusion, cell membrane deformation, imaging processing and geophysical
problems are formulated as second- or fourth-order partial differential equations  on arbitrary surfaces or 2D Riemannian manifolds \cite{BERESE07,DULIU04,DUJU09,WilDra92,DziEll13,Nitschke,Novak,Roy,Zhang2017}.
Generally speaking, numerical methods developed for solving such equations
\textcolor{blue}{can be classified into three  categories including embedded narrow-band methods, mesh-free particle method and intrinsic methods \cite{}.  The  narrow-band methods use level set method \cite{Adalsteinsson2003,Bergdorf2010,Bertalmio2001,Greer2006,MacDonald08,MacDonald09,Piret,Sbalzarini,Schwartz}, kernel method \cite{Fuselier,Schwartz}   or diffusive method \cite{Li2009,Ratz} to track the interface and the surface PDE is locally  expanded. The  mesh-free method employs Lagrangian framework on the surface \cite{LeuLow11,Suchdea}.
	The intrinsic methods make use of an intrinsic mesh to discretize the surface and then solve the surface PDE with various numerical methods. Finite volume schemes for solving conservation laws and parabolic equations on surfaces were introduced in
	\cite{Gie09,LefOku09,GieMul14} and \cite{LenNem11}, respectively.
	Authors of papers \cite{DUJU05,DUJU05A,DUJU09} presented mixed finite volume methods for solving surface convection-diffusion equations and higher-order PDE. In \cite{DUJU11}, the same authors described a finite element method for solving Cahn-Hilliard equations modeling a phase separation phenomenon on surfaces.  Finite element methods for solving various classes of surface PDEs were developed in
	\cite{Dzi88,DziEll07,Dzi08,DziEll13}.}

%

In this paper, we present local discontinuous Galerkin (LDG) methods \textcolor{blue}{based on the intrinsic method,} for solving convection-diffusion equation and Cahn-Hilliard equation on static surfaces, respectively, given that DG methods possess good stabilization mechanisms for
problems which can generate steep gradients or discontinuities in their solutions. The first LDG method for solving convection-diffusion equation was introduced in \cite{CocShu98}. Later, the LDG method for solving time-dependent PDEs with higher order derivatives was developed in \cite{YANSHU02}.  In \cite{XIAXU07}, A LDG method for solving Cahn-Hilliard type equations was presented.  Recent papers \cite{DedMad13,AntDed15} described several discontinuous Galerkin (DG) formulations for solving linear elliptic equations on surfaces,
and presented  a priori error estimates in the $L^2$ and energy norms for piecewise linear and higher-order
ansatz functions and surface approximations. The hybridizable DG method to solve elliptic problems on surfaces was analyzed in \cite{CocDem16}.   \textcolor{blue}{Xu et al. \cite{Guo2016,Xu2009} introduced LDG methods to solve the surface diffusion and Willmore flow on graphs (also see \cite{Aizinger,Boscarino}). These methods are mainly based on level set  representation,  and  the surface PDE is actually solved in 3D.}

\textcolor{blue}{The purpose of this work is to extend the LDG method to directly  solve surface PDEs posed on 2D manifolds \cite{Zhang2015}.}
The continuous surface $\G$ is approximated by a piecewise polygonal surface $\G_h$.   The essential ingredients of our LDG schemes consist of design of numerical fluxes motivated by \cite{CocShu98,AntDed15,DedMad15}, and the total variation diminishing (TVD) Runge-Kutta (RK) time stepping method \cite{GotShu98,ShuOsh88}. To this end, we show that these schemes are energy stable. Obviously, approximating $\G$ by $\G_h$ introduces a geometric
error.
We do not analyze how the geometric error affects the accuracy of the solution. Instead, we presented various numerical examples.  We demonstrate convergence of the fully discrete  schemes for solving convection-diffusion equations and results of using the scheme for solving the Cahn-Hilliard equations defined on general surfaces. We refer readers to \cite{DziEll13} and \cite{GieMul14} for an analysis of geometric errors, and \cite{DedMad13} for analyzing interior penalty (IP) DG method for elliptic
problems on surfaces. We also would like to point out that authors of \cite{DedMad15} derived a DG method for solving steady state solution of advection-dominated problems in which they used IP DG approach treating the diffusion term.

The rest of the paper is organized as follows:  Section~\ref{sec:Model}
describes the model equations considered in the paper.
Section~\ref{sec:ldg} is devoted to presenting implementation of the LDG schemes for solving the surface convection-diffusion  and Cahn-Hilliard equations.
Numerical tests are presented in Section~\ref{sec:NR}. Finally, conclusions are drawn in Section~\ref{sec:conclusion}.

\section{ Model Problems and Preliminaries}
\label{sec:Model}

We first set basic notations.  In this paper, we consider to solve initial value problems defined by two classes of equations posed on  a closed and oriented two-dimensional parameterizable $C^k$-surface
$\G \subset\mathbb{R}^3$.

For later use, we assume that there exists a sufficiently thin open subset $\sN \subset \mathbb{R}^3$ around $\G$ in a way that for every $\bx  \in \sN$ there is a unique point $\boldsymbol{\zeta}(\bx) \in \G $ with \cite{DziEll07,DziEll13,GieMul14,DedMad13}
\be
\bx = \boldsymbol{\zeta}(\bx) + d(\bx) \boldsymbol{\nu}_{\G}(\boldsymbol{\zeta}(\bx))~.
\label{eq:biject}
\ee
Here $d$ denotes the signed distance function to $\G$ which is assumed to be well defined in $\sN$, and $\boldsymbol{\nu}_{\G}(\boldsymbol{\zeta}(\bx))$ is the unit
normal vector to $\G$ pointing towards the non-compact component of $\mathbb{R}^3 \setminus \G $ \cite{DziEll07,GieMul14,DedMad13}.
We further assume that $\G$ is given locally by some diffeomorphism $X : \Omega \rightarrow \G  $ and some parameter domain
$\Omega \subset \mathbb{R}^2$ \cite{Dzi08}.

We use the notation $X = X(\theta)$ with $\theta \in \Omega$, and introduce standard notations \cite{DziEll13} for the first fundamental form
$G(\theta) = \left(g_{ij}(\theta) \right)_{i,j=1,2}$ by
\be
g_{ij}(\theta) = \frac{\partial X}{\partial \theta_i}(\theta) \cdot \frac{\partial X}{\partial \theta_j}(\theta)~.
\ee
We denote the determinant of $G$ by $g = \det (G)$, and use superscript indices to denote the inversion of the matrix $G$ so that
\be
\left(g^{ij} \right)_{i,j,=1,2} = G^{-1}~.
\ee

\begin{defn}
For a sufficiently smooth and differentiable function $f: \G \rightarrow R$, the tangential gradient in the parametric coordinates is given by
\be
 \nabla_\Gamma f (X(\theta)) = \sum_{i,j=1}^{2} g^{ij}(\theta) \frac{\partial f(X(\theta))}{\partial \theta_j}  \frac{\partial X}{\partial \theta_i}(\theta)~,
\ee
and the Laplace-Beltrami operator on $\G$ is defined by
\be
\Delta_{\Gamma} f= \frac{1}{\sqrt{g(\theta)}} \sum_{i,j=1}^{2} \frac{ \partial   }{\partial \theta_j} \left( g^{ij}(\theta) \sqrt{g(\theta)} \frac{\partial f(X(\theta)) }{\partial \theta_i} \right) ~.
\ee
\label{def:tang_lap}
\end{defn}

We refer readers to \cite{DziEll13} for definitions of these operators with respect to implicitly defined surfaces and related discussion.

In this paper, \textcolor{blue}{the following two types of equations on surfaces will be considered:
	\begin{itemize}
		\item  The time-dependent convection-diffusion equation, which is a second-order surface PDE:
		\be
		\begin{array}{ll}
			u_t +  \nabla_{\Gamma}\cdot (\bsw u)  = \nabla_{\Gamma} \cdot (a(u) \nabla_{\G} u)~,& ~~~~~~{\rm in} ~(0, T)\times \G~,\\
			u(t = 0 ) = u_0 ~, & ~~~~~~{\rm on} ~\G ~.
		\end{array}
		\label{eq:adv_diff}
		\ee
		where    $a(\cdot) > 0$ and the  velocity field $\bsw$ is  tangent to $\G$.
		\item The time-dependent Cahn-Hilliard equation, which is a fourth-order surface PDE:
		\be
		\begin{array}{ll}
			u_t   = \nabla_{\G} \cdot \left( b(u) \nabla_{\G} ( - \gamma \Delta_{\G} u + \Psi'(u) ) \right)~, & ~~~~~~{\rm in} ~(0, T)\times \G~,\\
			u(t = 0 ) = u_0 ~, & ~~~~~~{\rm on} ~\G ~.
		\end{array}
		\label{eq:bi_harmo}
		\ee
		where $\Psi(u) $ is the bulk homogeneous free  energy density. $b(u)$ is the diffusion mobility, and $\gamma$ is a positive constant.
	\end{itemize}
}

\section{Numerical Schemes}
\label{sec:ldg}

In the present work we use a triangulated   surface $\Gamma_h \subset \mathcal{N}$ composed of planar triangles $\bK_h$ whose vertices stand on $\Gamma$ to approximate $\Gamma$. Therefore,
$$\G_h = \bigcup_{\bK_h \in \mathcal{T}_h} \bK_h~,
$$
where $\mathcal{T}_h$ denotes  the set of the planar triangles which form an admissible triangulation.  Notice that there exists a one-to-one
relation between points $\bx \in \G_h$ and $ \boldsymbol{\zeta} \in \G$ satisfying Eq.~(\ref{eq:biject}).

Denote by $\mathcal{E}$ the set of edges (facets) of $\mathcal{T}_h$. For each $e \in \mathcal{E}$, denote by $h_e$ the length of the edge $e$. 
Let $N_{{\bK}_h}$ be an integer index of element ${\bK}_h$, and ${\bK}^{e,-}_{h}$ and ${\bK}^{e,+}_{h}$ be the two elements sharing the common edge $e$. 
Denote by $\bn^-_h$ and $\bn^+_h$  the unit outward conormal vectors defined on the edge $e$ for ${\bK}^{e,-}_{h}$ and ${\bK}^{e,+}_{h}$, respectively.  The conormal $\bn^-_h$ to a point $\bx \in e$ is defined as follows \cite{AntDed15} \textcolor{blue}{
	\begin{itemize}
		\item the unique unit vector that  lies in the plane containing  ${\bK}^{e,-}_{h}$;
		\item $
		\bn^-_h(\bx)\cdot (\bx - \by) \geq 0~,~ \forall \by \in {\bK}^{e,-}_{h} \cap B_\epsilon(\bx)~,
		$ where $B_\epsilon(\bx)$ is a ball centered in $\bx$.  The radius $\epsilon (>0)$ of $B_\epsilon(\bx)$ is sufficiently small.
	\end{itemize}
}
\noindent The conormal $\bn^+_h$ is defined similarly. With this definition, \textcolor{blue}{we have }
$$
\bn^+_h \neq -\bn^-_h~
$$
in general ( See Fig.~\ref{fig:surf_tri} for example).

\begin{figure}[!ht]
	\begin{center}
		\includegraphics[scale=0.6]{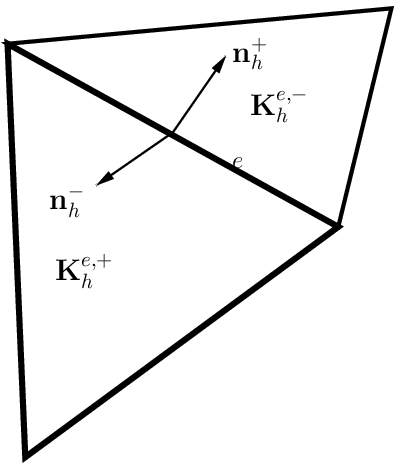}
		\caption{Two elements ${\bK}^{e,+}_{h}$ and ${\bK}^{e,-}_{h}$ and their respective conormals $\bn^+_h$ and $\bn^-_h$ on the common edge $e$. }
		\label{fig:surf_tri}
	\end{center}
\end{figure}

Let $\mathbb{P}_k(D)$ denote  the space of polynomials of degree not greater than $k$ on any planar domain $D$.
The discrete DG space $\mathbb{S}_{h,k}$ of scalar function associated with $\G_h$ is
$$
\mathbb{S}_{h,k} = \left\{ \chi \in L^2(\G_h) : \chi|_{\bK_h} \in \mathbb{P}_k(\bK_h), ~ \forall~ \bK_h \in \textcolor{blue}{\mathcal{T}_h} \right\}~,
$$
i.e. the space of piecewise polynomials which are globally in $L^2(\G_h)$.

The vector-valued DG space  $ \boldsymbol{\Sigma}_{h,k}$  associated with $\G_h$ is chosen to be
$$
\boldsymbol{\Sigma}_{h,k} = \left\{  \boldsymbol{\varphi} \in \left[L^2(\G_h)\right]^3 : \boldsymbol{\varphi} |_{\bK_h} \in  \left[\mathbb{P}_k(\bK_h) \right]^3 , ~ \forall~ \bK_h \in \textcolor{blue}{\mathcal{T}_h}\right\}~.
$$

For $ \upsilon_h \in \mathbb{S}_{h,k}$ and $\br_h \in \boldsymbol{\Sigma}_{h,k} $, we use $\upsilon_h^{\pm}$ and $\br_h^{\pm}$ to denote the trace of $ \upsilon_h $ and  $\br_h$ on $e = {\bK}^{e,+}_{h} \bigcap {\bK}^{e,-}_{h}$
taken within the interior of ${\bK}^{e,+}_{h} $ and ${\bK}^{e,-}_{h}$, respectively.

For a given function $f_h\in\mathbb{S}_{h,k}$, the surface gradient operator on $\G_h$ denoted by $\nabla_{\G_h}$ and the laplace-Beltrami operator
on $\G_h$ denoted by $\Delta_{\Gamma_h}$ are defined similarly as in Def.~\ref{def:tang_lap}.

For the purpose of analysis, we further introduce a lift onto $\G$ for any function $v_h$ defined
on the discrete surface $\G_h$ as follows.

\begin{defn} For any function $v_h \in \mathbb{S}_{h,k}$ defined on $\G_h$, the lift or extension \cite{DziEll13} onto $\G$ is given by:
	$$
	v^l_h(\boldsymbol{\zeta}) := v_h(\bx(\boldsymbol{\zeta}))~,~~~~~ \boldsymbol{\zeta} \in \G,
	$$
	where $\bx(\boldsymbol{\zeta})$ is the unique solution of Eq.~(\ref{eq:biject}).
	\label{def:lift}
\end{defn}

Also by  Def.~\ref{def:lift}, for every $\bK_h \subset\G_h$, there is a unique curved triangle $\bK^l_h =  \boldsymbol{\zeta}(\bK_h) \subset \G $.
When transforming from $\G_h$ to $\G$,  denote by $\delta_h$  the quotient between the smooth and discrete surface measures $dA$ on $\G$ and $dA_h$ on $\G_h$ so that $\delta_h dA_h = dA$, and $\delta_e$ the quotient between measures $ds_h$ on  $e$ and $ds$ on $e^l$ so that $\delta_e ds_h = ds$.

\textcolor{blue}{$\mathbb{S}^{l}_{h,k}$}, the lifted finite element space is defined as
\be
\textcolor{blue}{\mathbb{S}^l_{h,k} = \left\{ \phi_h = \varphi^l_h | \varphi_h \in  \mathbb{S}_{h,k}  \right\}}~.
\ee
The lifted space $\boldsymbol{\Sigma}^l_{h,k} $ is defined similarly.

\subsection{LDG Method for Surface Convection-Diffusion Equation}
\label{sec:ldg_conv_diff}
By introducing the
auxiliary variable $\mathbf{q} = \sqrt{a(u)} \nabla_{\Gamma} u$, the model problem (\ref{eq:adv_diff}) can be rewritten as a first order system of equations:
\be
\left\{
\begin{array}{l}
	u_t + \nabla_{\Gamma}\cdot (u\bsw - \sqrt{a(u)}  \mathbf{q}) = 0 ~, \\[3mm]
	\mathbf{q} - \nabla_{\Gamma} g(u) = 0 ~,
\end{array}
\right .
\label{eq:model_1st_order}
\ee
where $g(u) = \int^u \sqrt{a(s)}ds$.

The LDG method for solving Eqs.~(\ref{eq:model_1st_order}) is defined by: Find $u_h \in \mathbb{S}_{h,k}$ and
$\bq_h \in \boldsymbol{\Sigma}_{h,k}$, such that for all test functions $v_h \in \mathbb{S}_{h,k}$ and $\br_h \in \boldsymbol{\Sigma}_{h,k} $,
\be
\left\{
\begin{array}{l}
	\displaystyle{\int_{ \bK_h }} (u_h)_t v_h d\bx -
	\displaystyle{\int_{ \bK_h }} \left( u_h\bsw_h - \sqrt{a(u)}\bq_h  \right) \cdot \nabla_{\Gamma_h} v_h d\bx +
	\displaystyle{\int_{ \bK_h }} \left(  \beta_h u_h v_h \right) d\bx \\
	~~~~~~~~~~~ +\displaystyle{\int_{\partial \bK_h}} \reallywidehat{u_h\bsw_h} \cdot \bn_h v_h ds
	- \displaystyle{\int_{\partial \bK_h}} \reallywidehat{\sqrt{a} \bq} \cdot \bn_h v_h ds = 0 ~, \\[3mm]
	\displaystyle{\int_{\bK_h}} \bq_h \cdot \br_h d\bx = -\displaystyle{\int_{\bK_h}} g(u_h) \nabla_{\G_h} \cdot \br_h d\bx
	+ \displaystyle{\int_{\partial \bK_h}}  \widehat{g} \bn_h \cdot \br_h ds ~.
\end{array}
\right .
\label{eq:ldg}
\ee
Here $\widehat{g}$, $\reallywidehat{\sqrt{a} \bq}$ and $\reallywidehat{u_h\bsw_h}$ are numerical fluxes which \textcolor{blue}{will be } prescribed later on.
$\bsw_h$, \textcolor{blue}{defined  on $\G_h$, is a surface Raviart-Thomas interpolant of $\bsw$ such that normal jumps of the velocity across edges will vanish, namely, $\bsw_h^-\cdot\bn_h^-+\bsw_h^+\cdot\bn_h^+=0$.
The method to compute $\bsw_h$ is described in \cite{DedMad15}. The term containing $\beta_h$ is for stability, and $\beta_h$ is defined to be:
\be
\beta_h =  \left\{
\begin{array}{l}
	0~,~~~~~~~~~~~~~~~~~~{\rm if}~~  \nabla_{\G_h}\cdot \bsw_h \geq 0~;\\
	-\frac{1}{2}\nabla_{\G_h}\cdot \bsw_h~,~~~~{\rm if}~~  \nabla_{\G_h}\cdot \bsw_h < 0~.
\end{array}
\right .
\ee}

To facilitate  definitions of numerical fluxes, following trace operators $\{\cdot\}$ and $[\![\cdot ]\!]$ are introduced by following ideas in \cite{AntDed15}.
\begin{defn} Denote by ${\bK}^{e,+}_{h}$ and ${\bK}^{e,-}_{h}$  the two elements sharing the common edge $e$.
	\textcolor{blue}{
		\begin{itemize}
			\item For $\upsilon \in  L^2(\G_h)$,
			$\{\upsilon\}$ and $[\![ \upsilon ]\!] $ are defined as
			\be
			\{ \upsilon \} = \frac{1}{2}(\upsilon^- + \upsilon^+ )~,
			~~[\![ \upsilon ]\!] = \upsilon^+- \upsilon^-~ { on }~e.
			\ee
			\item For $\boldsymbol{\varphi} \in [ L^2(\G_h)]^3 $, $\{ \boldsymbol{\varphi},\bn_h \}$ and
			$[\![ \boldsymbol{\varphi} ,\bn_h ]\!]$ are defined as
			\be
			\{ \boldsymbol{\varphi},\bn_h \}  = \frac{1}{2} ( \boldsymbol{\varphi}^+\cdot \bn^+_h - \boldsymbol{\varphi}^-\cdot \bn^-_h )~,
			~~[\![ \boldsymbol{\varphi} ,\bn_h ]\!] = \boldsymbol{\varphi}^- \cdot \bn^-_h + \boldsymbol{\varphi}^+ \cdot \bn^+_h ~{ on }~e.
			\ee
		\end{itemize}
	}
\end{defn}

Denote by $S^{{\bK}^-_h}_{{\bK}^+_h} \in \{0, 1 \}$ a switch function \cite{PERPER08}.  $S^{{\bK}^-_h}_{{\bK}^+_h}$ is associated with $\bK^+_h$ on the edge that $\bK^+_h$ and $\bK^-_h$ share, and is defined by:
\be
S^{{\bK}^-_h}_{{\bK}^+_h} =
\left\{
\begin{array}{l}
	1~,~~~~ {\rm if}~ N_{{\bK}_h^+} > N_{{\bK}_h^-}~, \\
	0~,~~~~ {\rm otherwise}~.
\end{array}
\right .
\ee


Motivated by \cite{CocShu98,AntDed15,DedMad15}, the numerical fluxes \textcolor{blue}{on $e$ for ${\bK}^{e,+}_{h} $ and ${\bK}^{e,-}_{h}$
are defined} respectively,  as follows.
\textcolor{blue}{
	\begin{itemize}
		\item The convective fluxes   $\reallywidehat{u_h\bsw_h}^+$ and $\reallywidehat{u_h\bsw_h}^-$ :
		\be
		\begin{array}{lll}
			\reallywidehat{u_h\bsw_h}^+ & = & ~~ \left(\{u_h\bsw_h, \bn_h \} + \frac{\alpha}{2} [\![ u_h ]\!]\right)\bn^+_h~,\\
			\reallywidehat{u_h\bsw_h}^- & = & -\left(\{u_h\bsw_h, \bn_h \} + \frac{\alpha}{2} [\![ u_h ]\!] \right)\bn^-_h~,
		\end{array}
		\label{eq:adv_flux}
		\ee
		where $\alpha = \max |\bn_h \cdot \bsw_h|$.
		\item  The diffusive fluxes $\left(\reallywidehat{\sqrt{a}\bq}^+ ,  \widehat{ g }^+ \right)^T$ and
		$\left(\reallywidehat{\sqrt{a}\bq}^- ,  \widehat{ g}^- \right)^T$  :
		\be
		\begin{array}{lll}
			\reallywidehat{\sqrt{a}\bq}^+ &=& ~~\left( \frac{[\![g(u_h)]\!] }{[\![u_h]\!]} \{  \bq_h, \bn_h \} - C_{11}  [\![ u_h ]\!] + \boldsymbol{C}_{12}\cdot \bn_h^+[\![ \bq_h,\bn_h ]\!]\right)\bn^+_h~\\
			\reallywidehat{\sqrt{a}\bq}^- &=& -\left(\frac{[\![gg(u_h)]\!] }{[\![u_h]\!]} \{ \bq_h, \bn_h \} - C_{11} [\![ u_h ]\!] + \boldsymbol{C}_{12}\cdot \bn_h^+[\![ \bq_h,\bn_h ]\!]\right)\bn^-_h~.
		\end{array}
		\label{eq:diff_flux_q}
		\ee
		\be
		\begin{array}{lll}
			\widehat{ g }^+  &=& \{  g(u_h) \} - \boldsymbol{C}_{12} \cdot \bn^+_h [\![ u_h ]\!]~,\\
			\widehat{ g }^- &=& \{  g(u_h) \} - \boldsymbol{C}_{12} \cdot \bn^+_h [\![ u_h ]\!]~.
		\end{array}
		\label{eq:diff_flux_u}
		\ee
	Here the penalization coefficients $C_{11}$ and $\boldsymbol{C}_{12}$ are chosen to be
		$$
		C_{11} = \frac{1}{h_e}~,~~~~~~~ \boldsymbol{C}_{12} = \frac{1}{2}\left( S^{{\bK}_h^{e,-}}_{{\bK}_h^{e,+}} \bn^+_h + S^{{\bK}_h^{e,+}}_{{\bK}_h^{e,-}} \bn^-_h  \right)~.
		$$
	\end{itemize}
}



It is easy to see that with the above choices of the numerical fluxes, $[\![ \reallywidehat{u_h\bsw_h} , \bn  ]\!] = 0 $,  $[\![ \widehat{g} ]\!] = 0$ and $[\![ \reallywidehat{\sqrt{a}\bq}, \bn ]\!] = 0$. Thus these numerical fluxes  are consistent and  conservative. Moreover, they allow for a local resolution of $\bq_h$ in terms of $u_h$.

\textcolor{blue}{We introduce the surface $H^1$ space   $H^1(\G_h) = \left\{ v|_{\bK_h} \in H^1(\bK_h), \forall ~ \bK_h \in \G_h \right\}$. The following Lemma is needed for analyzing the $L^2$-stability of the scheme (\ref{eq:ldg}).  }
\begin{lma} \textcolor{blue}{(\cite{AntDed15})}
	Let $v\in H^1(\G_h)$ and $\br \in [H^1(\G_h)]^3$. Then
	\be
	\begin{array}{lll}
		\displaystyle{\sum_{\bK_h \in \G_h}} \displaystyle{\int_{ \partial \bK_h }} v\br \cdot \bn_h ds & = &  \displaystyle{\sum_{e \in \mathcal{E}} } \displaystyle{\int_{ e }} [\![ \br, \bn_h ]\!] \{ v \} + \{\br, \bn_h \} [\![ v ]\!] ds~.
	\end{array}
	\ee
	\label{lma:identity}
\end{lma}


The following estimate holds for the LDG scheme (\ref{eq:ldg}) to solve the surface convection-diffusion equation (\ref{eq:adv_diff}).
\begin{proposition}[$L^2$-stability of LDG scheme (\ref{eq:ldg})]
	Let $(u_h, \bq_h)^T$ be a solution of the scheme (\ref{eq:ldg}) with  initial value $u_h(\bx,t=0)$ given by $L^2$ projection of $u_0$,
	and $(u^l_t, \bq^l_h)^T$ its lift according to Definition~\ref{def:lift}. Then the following stability estimate holds:
	\begin{eqnarray}
	&&\frac{1}{2}\int_{\G} \frac{1}{\delta_h} (u_h^l)^2(\bx,T)d\bx +\int^T_0\int_{\G} \frac{1}{\delta_h} |\bq^l_h(\bx,t)|^2d\bx dt + \Theta_{T,C}([\sW^l_h])\nonumber\\ &\leq& \frac{1}{2}\int_{\G} \frac{1}{\delta_h} (u_0)^2(\bx)d\bx~,
	\end{eqnarray}
	where $\sW^l_h = (u^l_h, \bq^l_h)^T$ and  $\Theta_{T,C}([\sW^l_h]) = \displaystyle{\int_{0}^{T}} \displaystyle{\sum_{e^l \in \mathcal{E}^l} } \displaystyle{\int_{ e^l }} \frac{1}{\delta_e} \left(C_{11}+\frac{\alpha}{2}\right)[\![ u^l_h ]\!]^2 ds dt$.
	\label{prop:ldg_diff_stab}
\end{proposition}
Proof of Proposition \ref{prop:ldg_diff_stab} follows directly from \cite{CocShu98,DziEll07,DedMad15}, and is briefly described below.

\begin{proof}
	Denote by $\sW_h = (u_h, \bq_h)^T$ and $\sV_h = (v_h, \br_h)^T$.
	Let's first define a functional  by  adding two component equations of  Eq.~(\ref{eq:ldg}) up, and  summing over all elements:

	\begin{eqnarray} 	\label{eq:ldg_Ah_diff}
		0&=&\sA_h(\sW_h, \sV_h)\nonumber \\
		&=& \displaystyle{\sum_{\bK_h \in \G_h}} \displaystyle{\int_{ \bK_h} (u_h)_t v_h d\bx}
		-\displaystyle{\sum_{\bK_h \in \G_h}} \displaystyle{\int_{\bK_h}} u_h\bsw_h\cdot\nabla_{\Gamma_h} v_h d\bx
		+ \displaystyle{\sum_{\bK_h \in \G_h}} \displaystyle{\int_{ \bK_h }} \bq_h \cdot \br_h d\bx\nonumber  \\
		&& +  \displaystyle{\sum_{\bK_h \in \G_h}} \displaystyle{\int_{ \bK_h }}  \sqrt{a}\bq_h  \cdot \nabla_{\Gamma_h} v_h  d\bx + \displaystyle{\sum_{\bK_h \in \G_h}} \displaystyle{\int_{ \bK_h }} g(u_h) \nabla_{\G_h} \cdot \br_h d\bx \nonumber \\
		&& +  \sumK\intK \beta_h u_h v_h d\bx  +  \displaystyle{\sum_{\bK_h \in \G_h}} \displaystyle{\int_{\partial \bK_h }} v_h \reallywidehat{u_h \bsw_h}\cdot \bn_h ds\nonumber  \\
		& &+  \displaystyle{\sum_{\bK_h \in \G_h}} \displaystyle{\int_{\partial \bK_h }} (-\reallywidehat{\sqrt{a}\bq})\cdot \bn_h v_h ds -
		\displaystyle{\sum_{\bK_h \in \G_h}} \displaystyle{\int_{\partial \bK_h }}\hat{g}\bn_h\cdot\br_h ds~.
	\end{eqnarray}

	
	By replacing $v_h(\bx)$ and $\br_h(\bx)$ by $u_h(\bx, t)$ and $\bq_h(\bx,t)$ in Eq.~(\ref{eq:ldg_Ah_diff}), integrating from 0 to $T$,
	and using Lemma~\ref{lma:identity} and Eqs.~(\ref{eq:adv_flux}) - (\ref{eq:diff_flux_u}), it can be shown that
	\be
	\begin{array}{lll}
		0&= &\intT \sA_h(\bw_h, \bw_h)dt \\
		&=&
		\intT \sumK \displaystyle{\int_{ \bK_h } (u_h)_t u_h d\bx}dt +
		\intT \sumK \displaystyle{\int_{ \bK_h }} |\bq_h|^2 d\bx dt \\
		& - & \intT \displaystyle{\sum_{\bK_h \in \G_h}} \displaystyle{\int_{\bK_h}} u_h\bsw_h\cdot\nabla_{\Gamma_h} u_h d\bx dt + \intT \sumK \displaystyle{\int_{\partial \bK_h }} u_h \reallywidehat{u_h \bsw_h}\cdot \bn_h dsdt  ~\\
		&+&  \intT \sume \displaystyle{\int_{e}} C_{11}[\![ u_h ]\!]^2 ds dt +\intT\sumK\intK \beta_h u^2_h d\bx dt  ~ \\
		&=& \frac{1}{2}\displaystyle{\int_{\G_h} (u_h)^2(\bx, T) d\bx} dt - \frac{1}{2}\displaystyle{\int_{ \G_h } (u_h)^2(\bx, 0) d\bx} dt  \\
		&& + \intT\displaystyle{\int_{\G_h}} |\bq_h|^2 d\bx dt +
		\intT\sume \displaystyle{\int_{e}} C_{11}[\![ u_h ]\!]^2 ds dt \\
		&+&\intT\sumK\intK \beta_h u^2_h d\bx dt-\intT\sumK \intK u_h\bsw_h\cdot\nabla_{\Gamma_h} u_h d\bx dt~ \\
		&+&\intT\sumK \intpK u_h \reallywidehat{u_h \bsw_h}\cdot \bn_h dsdt
	\end{array}
	\ee
	
	Define by
	\begin{eqnarray}
	&&\sB_h(\sW_h, \sW_h)\nonumber\\
	&= &-\intT \sumK \intK u_h\bsw_h\cdot\nabla_{\Gamma_h} u_h d\bx dt
	+ \intT \sumK \intpK u_h \reallywidehat{u_h \bsw_h}\cdot \bn_h ds dt.\nonumber
	\end{eqnarray}
	
	By integrating by parts on each $\bK_h$ for the first term of $\sB_h(\sW_h, \sW_h)$ and using $[\![ \bsw_h, \bn_h ]\!] = 0$ and Eq.~(\ref{eq:adv_flux}), we obtain

	\begin{eqnarray}
		&&\sB_h(\sW_h, \sW_h)\nonumber\\
		 &=& -\frac{1}{2}\intT \sumK \intpK (u_h)^2 \bsw_h\cdot \bn_h dsdt
		+\frac{1}{2}\intT \sumK\intK u_h^2 \nabla_{\Gamma_h}\cdot\bsw_h d\bx dt \nonumber \\
		&&+ \intT\sumK\intpK u_h \reallywidehat{u_h \bsw_h}\cdot \bn_h ds dt\nonumber  \\
		&=& -\frac{1}{2}\intT \sume \displaystyle{\int_e}( [\![\bsw_h,\bn_h]\!] \{u_h^2 \} + \{\bsw_h, \bn_h \}[\![u_h^2]\!])dsdt\nonumber  \\
		&&+\intT\sume \displaystyle{\int_e}( \{u_h\bsw_h, \bn_h\} + \frac{\alpha}{2}[\![u_h]\!])[\![u_h]\!] dsdt
		+\frac{1}{2}\intT \sumK\intK u_h^2 \nabla_{\Gamma_h}\cdot\bsw_h d\bx dt~\nonumber \\
		&=& -\frac{1}{2}\intT \sume \displaystyle{\int_e}(\{\bsw_h, \bn_h \}[\![u_h^2]\!])dsdt   \\
		&&+ \intT \sume \displaystyle{\int_e}(\frac{1}{2}\{\bsw_h, \bn_h \}[\![u_h^2]\!] + \frac{\alpha}{2}[\![u_h]\!]^2)dsdt
		+\frac{1}{2}\intT \sumK\intK u_h^2 \nabla_{\Gamma_h}\cdot\bsw_h d\bx dt~.\nonumber
	\end{eqnarray}

	This leads to
	\be
	\begin{array}{lll}
		\intT \sA_h(\bw_h, \bw_h)dt &=& \frac{1}{2}\displaystyle{\int_{\G_h} (u_h)^2(\bx, T) d\bx} dt -
		\frac{1}{2}\displaystyle{\int_{ \G_h } (u_h)^2(\bx, 0) d\bx} dt \\
		& +& \intT\displaystyle{\int_{\G_h}} |\bq_h|^2 d\bx dt +
		\intT\sume \displaystyle{\int_{e}} (\frac{\alpha}{2} + C_{11})[\![ u_h ]\!]^2 ds dt~.
	\end{array}
	\ee

	Since $\displaystyle{\int_{ \G_h } (u_h)^2(\bx, 0) d\bx} \leq  \displaystyle{\int_{ \G_h } (u_0)^2 d\bx} $ (by Lemma 4.2 in \cite{DziEll13}), it gives rise to:
	\begin{eqnarray}
	&&\frac{1}{2} \displaystyle{\int_{ \G_h } (u_h)^2(\bx, T) d\bx} +
	\intT\displaystyle{\int_{\G_h}} |\bq_h|^2 d\bx dt +
	\intT\sume \displaystyle{\int_{e}} (\frac{\alpha}{2} + C_{11})[\![ u_h ]\!]^2 ds dt\nonumber\\
	&\leq& \frac{1}{2} \displaystyle{\int_{ \G_h } (u_h)^2(\bx, 0) d\bx} ~.
	\end{eqnarray}
	
	By lifting $u_h$ and $\bq_h$ onto surface $\G$, it is shown that
	\begin{eqnarray}
	&&\frac{1}{2} \displaystyle{\int_{ \G } \frac{1}{\delta_h} (u^l_h)^2(\bx, T) d\bx} +
	\intT\displaystyle{\int_{\G}} \frac{1}{\delta_h} |\bq^l_h|^2 d\bx dt +
	\intT\sume \displaystyle{\int_{e}} \frac{1}{\delta_e} (\frac{\alpha}{2} + C_{11})[\![ u_h ]\!]^2 ds dt\nonumber\\
	&\leq& \frac{1}{2} \displaystyle{\int_{ \G } \frac{1}{\delta_h} (u_0)^2(\bx, 0) d\bx} ~.
	\end{eqnarray}

\end{proof}

\subsection{LDG Method for Cahn-Hilliard Equation}
\label{sec:ldg_bi_harmo}

In order to solve the model problem (\ref{eq:bi_harmo}) by the LDG method, the equation is rewritten as a first order system of \textcolor{blue}{equations}:
\be
\left\{
\begin{array}{l}
	u_t   =  \nabla_{\G} \cdot \mathbf{S} ~, \\
	\bS = b(u)\bP ~,\\
	\bP = \nabla_{\G} (-q + r )~, \\
	q = \gamma \nabla_{\G} \cdot \bW~, \\
	\bW = \nabla_{\G} u~, \\
	r =  \Psi'(u) ~.
\end{array}
\right .
\label{eq:bi_harmo_1st_order}
\ee

The LDG scheme for solving Eqs.~(\ref{eq:bi_harmo_1st_order}) is as follows: Find $u_h, q_h, r_h \in \mathbb{S}_{h,k}$ and
$\bS_h, \bP_h, \bW_h \in \boldsymbol{\Sigma}_{h,k}$, such that for all test functions $v_h, \varphi_h, \xi_h \in \mathbb{S}_{h,k}$ and
$\mathbf{\Theta}_h, \mathbf{\Phi}_h, \mathbf{\Upsilon}_h \in \boldsymbol{\Sigma}_{h,k}$,
\be
\left\{
\begin{array}{l}
	\displaystyle{\int_{{\bK}_h}} (u_h)_t v_h d\bx +
	\displaystyle{\int_{{\bK}_h}}  \bS_h \cdot \nabla_{\G_h} v_h d\bx -
	\displaystyle{\int_{\partial {\bK}_h}}    \widehat{ \bS} \cdot \bn_h v_h ds  = 0 ~, \\
	
	\displaystyle{\int_{{\bK}_h}} \bS_h \cdot \mathbf{\Theta}_h d\bx = \displaystyle{\int_{{\bK}_h}} b(u_h)\bP_h  \cdot \mathbf{\Theta}_h d\bx ~, \\
	
	\displaystyle{\int_{{\bK}_h}} \bP_h \cdot \mathbf{\Phi}_h d\bx = -\displaystyle{\int_{{\bK}_h}} (-q_h+r_h)
	\nabla_{\G_h} \cdot \mathbf{\Phi}_h d\bx +
	\displaystyle{\int_{\partial {\bK}_h}} (\hat{r}-\hat{q})\bn_h \cdot \mathbf{\Phi}_h ds~, \\
	
	\displaystyle{\int_{{\bK}_h}} q_h \varphi_h d\bx = -\gamma \displaystyle{\int_{{\bK}_h}} \bW_h \cdot
	\nabla_{\G_h} \varphi_h d\bx +
	\gamma\displaystyle{\int_{\partial {\bK}_h}}  \widehat{\bW} \cdot \bn_h \varphi_h ds~, \\
	
	\displaystyle{\int_{{\bK}_h}} \bW_h \cdot \mathbf{\Upsilon}_h d\bx = -\displaystyle{\int_{{\bK}_h}} u_h \nabla_{\G_h} \cdot \mathbf{\Upsilon}_h d\bx
	+ \displaystyle{\int_{\partial {\bK}_h}} \widehat{u} \bn_h \cdot \mathbf{\Upsilon}_h ds ~\\
	
	\displaystyle{\int_{{\bK}_h}} r_h  \xi_h d\bx = \displaystyle{\int_{{\bK}_h}} \Psi'(u_h) \xi_h d\bx~.
\end{array}
\right .
\label{eq:ldg_CH}
\ee

Following ideas in \cite{CocShu98,AntDed15,XIAXU07} the numerical fluxes on $e$ for $\bK_h^{e,+}$ and $\bK_h^{e,-}$ for solving the system (\ref{eq:ldg_CH}) are defined by:
\be
\begin{array}{lll}
	\widehat{ \bS}^+ & = & ~~\left(\{  \bS_h, \bn \} - C_{11} [\![ u_h ]\!] + \boldsymbol{C}_{12}\cdot \bn_h^+[\![ \bS_h,\bn ]\!]\right)\bn^+_h~,\\
	\widehat{ \bS}^- & = & -\left(\{  \bS_h, \bn \} - C_{11} [\![ u_h ]\!] + \boldsymbol{C}_{12}\cdot \bn_h^+[\![ \bS_h,\bn ]\!]\right)\bn^-_h~.
\end{array}
\ee

\be
\begin{array}{lll}
	\widehat{ u }^+ & = & \{ u_h \} - \boldsymbol{C}_{12} \cdot \bn^+_h [\![ u_h ]\!]~\\
	\widehat{ u }^- & = & \{ u_h \} - \boldsymbol{C}_{12} \cdot \bn^+_h [\![ u_h ]\!]~.
\end{array}
\ee

\be
\begin{array}{lll}
	\widehat{ \bW}^+ &=& ~~\left(\{  \bW_h, \bn \} - C_{11} [\![ u_h ]\!] + \boldsymbol{C}_{12}\cdot \bn_h^+[\![ \bW_h,\bn ]\!]\right)\bn^+_h~,\\
	\widehat{ \bW}^- &=& -\left(\{  \bW_h, \bn \} - C_{11} [\![ u_h ]\!] + \boldsymbol{C}_{12}\cdot \bn_h^+[\![ \bW_h,\bn ]\!]\right)\bn^-_h~.
\end{array}
\ee

\be
\begin{array}{lll}
	\widehat{ r }^+ &=& \{  r_h \} - \boldsymbol{C}_{12} \cdot \bn^+_h [\![ r_h ]\!]~,\\
	\widehat{ r }^- &=& \{  r_h \} - \boldsymbol{C}_{12} \cdot \bn^+_h [\![ r_h ]\!]~.
\end{array}
\ee

\be
\begin{array}{lll}
	\widehat{ q }^+ = \{  q_h \} - \boldsymbol{C}_{12} \cdot \bn^+_h [\![ q_h ]\!]~,\\
	\widehat{ q }^- = \{  q_h \} - \boldsymbol{C}_{12} \cdot \bn^+_h [\![ q_h ]\!]~.
\end{array}
\ee

\begin{proposition}[energy stability of LDG scheme (\ref{eq:ldg_CH})]
	The solution to the scheme (\ref{eq:ldg_CH}) satisfies the energy stability
	$$
	\frac{d}{dt}\displaystyle{\int_{\G}}\frac{1}{\delta_h} \left(\frac{\gamma}{2} \bW^l_h\cdot\bW^l_h + \Psi(u^l_h)\right)d\bx \leq 0~.
	$$
	\label{prop:LDG_CH_engy_stab}
\end{proposition}

\begin{proof}
	Following the proof in \cite{XIAXU07},
	it can be shown that
	\be
	\frac{d}{dt}\displaystyle{\int_{{\G}_h}} (\frac{\gamma}{2} \bW_h\cdot\bW_h + \Psi(u_h))d\bx \leq 0~.
	\ee
	By lifting $u_h$  and $\bW_h$ onto surface $\G$, it yields:
	\be
	\frac{d}{dt}\displaystyle{\int_{\G}}\frac{1}{\delta_h} \left(\frac{\gamma}{2} \bW^l_h\cdot\bW^l_h + \Psi(u^l_h)\right)d\bx \leq 0~.
	\ee
\end{proof}


\subsection{ Time discretization}
\label{sec:time_stepping}

The method of lines spatial approximation to PDEs is used for discretization and the time $t$ is left to be continuous till now. The second-order accurate TVD RK time discretization \cite{ShuOsh88,GotShu98}
is used to solve the semi-discrete schemes (\ref{eq:ldg}) and (\ref{eq:ldg_CH}), which can be formulated as an ordinary differential equation:
\be
\Phi_t = L(\Phi, t)~.
\label{eq:ODE}
\ee
The second-order accurate TVD RK method for solving Eq.~(\ref{eq:ODE}) is given by
\be
\begin{array}{lll}
\Phi^{(1)} & = &\Phi^n + \Delta t_n L(\Phi^n, t_n)~, \\
\Phi^{n+1} &= &\frac{1}{2}\Phi^n + \frac{1}{2} \Phi^{(1)} + \frac{1}{2}\Delta t_n L(\Phi^{(1)}, t_{n+1})~.
\end{array}
\ee
Here $\Delta t_n$ is the time step size.

\begin{rmk}
	In this paper, we only considered planar triangulations of the surface $\G$, which is at most second-order accurate. Therefore, we choose
	the second-order accurate TVD RK time-stepping method. Higher-order accurate surface
	approximation is needed in order to improve the overall accuracy of the schemes.
\end{rmk}

\section{Numerical results}
\label{sec:NR}

Numerical experiments based on the schemes defined in Sec. 3 are presented in this section. All meshes used here are generated by
MATLAB program DistMesh \cite{PerStr04}. The CFL condition for LDG scheme for convection-diffusion equation is
\be
\min \left(\frac{ \max|\bsw_h| \Delta t_n}{h_e},   \frac{ \max(a'(u)) \Delta t_n}{h_e^2} \right) \leq {\rm C}_{CFL}~,
\ee
${\rm C}_{CFL}$ is the CFL number. The time step size for LDG scheme for Cahn-Hilliard equation is:
\be
\Delta t_n = \mathcal{O}(h_e^4)~.
\ee

\subsection{Linear advection on a sphere}

\textcolor{blue}{We first illustrate the convergence rate of the numerical scheme by using the}
\textbf{test problem 6} described in \cite{GieMul14}. This test problem solves
\be
\frac{\partial u}{\partial t} + \nabla_\G\cdot (u\mathbf{V}) = 0~, ~~~~~~ t\in(0, 1]~,
\label{eq:test6}
\ee
defined on a unit spherical surface given by
$$
\mathbb{S}_1 = \left\{ \mathbf{x}\in \mathbb{R}^3 ~|~ x_1^2 + x_2^2 + x_3^2 = 1 \right\}~.
$$

The divergence-free advecting velocity field $\mathbf{V} = (v_1, v_2, v_3)^T$ at position $\bx = (x_1, x_2, x_3)^T$ is given by
$\mathbf{V} = \frac{2\pi}{||\bx||}(x_2, -x_1, 0)^T$.
The  initial condition is given by
\be
u_0(\bx) =\left \{
\begin{array}{ll}
	\frac{1}{10} \exp\left( \frac{-2(1+r^2(\bx))}{(1-r^2(\bx))^2} \right)  & ~~{\rm if}~r(\bx) < 1~, \\
	0  &  ~~{\rm otherwise}
\end{array}
\right .
\ee
with $r(\bx) = \frac{|\bx_0 - \bx|}{0.74}$ and $\bx_0 = (1, 0, 0)^T$. Table \ref{tab:lin_adv_test6}
shows that the second-order accuracy of the numerical solution is achieved.

\begin{table}[h!b!p!]
	\caption{Numerical errors and convergence order for $u_h$  of the DG method for solving the problem (\ref{eq:test6}).}
	\begin{center}
		\begin{tabular}{|l|l|l|l|l|}
			\hline
			h & {$L_1$ }& order & {$L_\infty$} & order  \\
			\hline
			0.2 &  1.21E-3 & -  & 3.77E-3  & -   \\
			\hline
			0.1 &  3.04E-4 & 1.99 & 7.70E-4 & 2.29  \\
			\hline
			0.05 & 5.63E-5 & 2.43 & 1.74E-4 & 2.15 \\
			\hline
			0.025 & 1.00E-5 & 2.49 & 3.43E-5 & 2.34 \\
			\hline
		\end{tabular}
	\end{center}
	\label{tab:lin_adv_test6}
\end{table}

\subsection{Diffusion on a sphere}

We next simulate the diffusion of a material on a sphere of radius $r$ and centered at the origin of the spherical coordinate specified by
$\phi$, $\theta$ and $r = 1$ \cite{LeuLow11}.
The initial condition and the exact solution to the diffusion equation on the surface
\be
\frac{\partial f}{\partial t} = \Delta_{\G} f ~,~~~~~~~~ t\in (0, 0.02]
\label{eq:diff_test}
\ee
is given by
\be
f(\phi, \theta, t) = \left(\sin^5(\theta)\cos(5\phi) + \sin^4(\theta)\cos(\theta)\cos(4\phi)\right)\exp(-30t/r^2)~.
\ee

Table~\ref{tab:lin_diff_test} shows the numerical errors and the second-order of accuracy of the solution.

\begin{table}[h!b!p!]
	\caption{Numerical errors and convergence order of the LDG method for solving Eq.~(\ref{eq:diff_test}).}
	\begin{center}
		\begin{tabular}{|l|l|l|l|l|l|l|}
			\hline
			h & {$L_1$ }& order & {$L_\infty$} & order & {$L_2$ }& order   \\
			\hline
			0.1 & 2.93E-2 & -  & 2.02E-2  & - & 1.23E-2 & -  \\
			\hline
			0.05 & 7.13E-3 & 2.04 & 5.73E-3 & 1.82 & 3.02E-3 & 2.03 \\
			\hline
			0.025 & 1.75E-3 & 2.03 & 1.36E-3 & 2.07 & 7.43E-4 & 2.02 \\
			\hline
			0.0125 & 4.38E-4 & 2.00 & 3.74E-4 & 1.86 & 1.85E-4 & 2.01 \\
			\hline
		\end{tabular}
	\end{center}
	\label{tab:lin_diff_test}
\end{table}

\subsection{Convection-diffusion on a sphere}
Motivated by test problems in \cite{DUJU05}, we consider to solve on the unit sphere $\mathbb{S}_1$ convection-diffusion problem
described by
\be
\frac{\partial u}{\partial t} + \nabla_\G \cdot (u\mathbf{V})  = a\Delta_{\G} u + s ~.
\label{eq:adv_diff_samp}
\ee
The divergence-free advecting velocity field $\mathbf{V} = (v_1, v_2, v_3)^T$ at position $\bx = (x_1, x_2, x_3)^T$ is given by
$\mathbf{V} = \frac{2\pi}{||\bx||}(x_2, -x_1, 0)^T$. In the spherical coordinate system $(\theta, \phi)$, defined by
$$
\bx = (r \sin\theta\cos\phi, r\sin\theta\sin\phi, r\cos\theta)^T
$$
for $\theta\in[0, \pi]$ and $\phi\in [0,2\pi)$,
the exact solution to this problem is chosen to be
$$
u(\bx,t) = \exp(-t)\sin^2(\theta)\cos^2(\phi)~.
$$
Value of the diffusion coefficient $a$ is set to be 0.05. The source term function $s$ is chosen such that $u(\bx,t)$ satisfies Eq.~(\ref{eq:adv_diff_samp}) exactly. Table~\ref{tab:conv_diff_test} shows the numerical errors and the second-order accuracy of the numerical solution to problem (\ref{eq:adv_diff_samp}).

\begin{table}[!htp]
	\caption{Numerical errors and convergence order of the LDG method for solving Eq.~(\ref{eq:adv_diff_samp}).}
	\begin{center}
		\begin{tabular}{|l|l|l|l|l|l|l|}
			\hline
			h & {$L_1$ }& order & {$L_\infty$} & order    \\
			\hline
			0.1 &  5.67E-3 & -  & 1.35E-3  & -  \\
			\hline
			0.05 & 1.36E-3 & 2.06 & 3.48E-4 & 1.96  \\
			\hline
			0.025 & 3.35E-4 & 2.02 & 8.98E-5 & 1.95  \\
			\hline
			0.0125 & 8.33E-5 & 2.01 & 2.28E-5 & 1.98  \\
			\hline
		\end{tabular}
	\end{center}
	\label{tab:conv_diff_test}
\end{table}

\subsection{Cahn-Hilliard equation on surfaces}

In this set of numerical simulations, we set $ \Psi'(u) = u^3 - u$, $b(u)=1.0$, and $\gamma = 0.008$ as in \cite{DUJU11}. The tests are done on different surfaces. The initial condition for all simulations is randomly generated.
The first simulation is performed on the unit spherical surface $\mathbb{S}_1$;
the second simulation is done on an ellipsoid defined by
$$
\mathbb{S}_2 = \left\{ \mathbf{x}\in \mathbb{R}^3 ~|~ \frac{x_1^2}{4} + \frac{x_2^2}{1} + \frac{x_3^2}{1.5^2} = 1 \right\}~;
$$
and the third simulation is performed on a biconcave surface defined by
$$
\mathbb{S}_3 = \left\{ \mathbf{x} \in \mathbb{R}^3 ~|~  x_3^2 - 0.25 R_0^2 \left(1- \frac{x_1^2 + x_2^2}{R_0^2} \right)
\left(C_0 + C_1 \frac{x_1^2 + x_2^2}{R_0^2} + C_2 \frac{ x_1^2 + x_2^2}{R_0^2} \right)^2   = 0 \right \}~,
$$
where $R_0 = 1.4$, $C_0= 0.207161$, $C_1 = 2.002558$ and  $C_2 = -1.122762$.

Figures \ref{fig:sphere_CH}, \ref{fig:ellip_CH} and \ref{fig:rbc_CH} show simulation results. We can observe that initially randomly generated patterns gradually merge  into large structures as expected.

\begin{figure}[!ht]
	\centering
	\subfigure[]{\label{fig:sphere_CH_t_0}\epsfig{figure=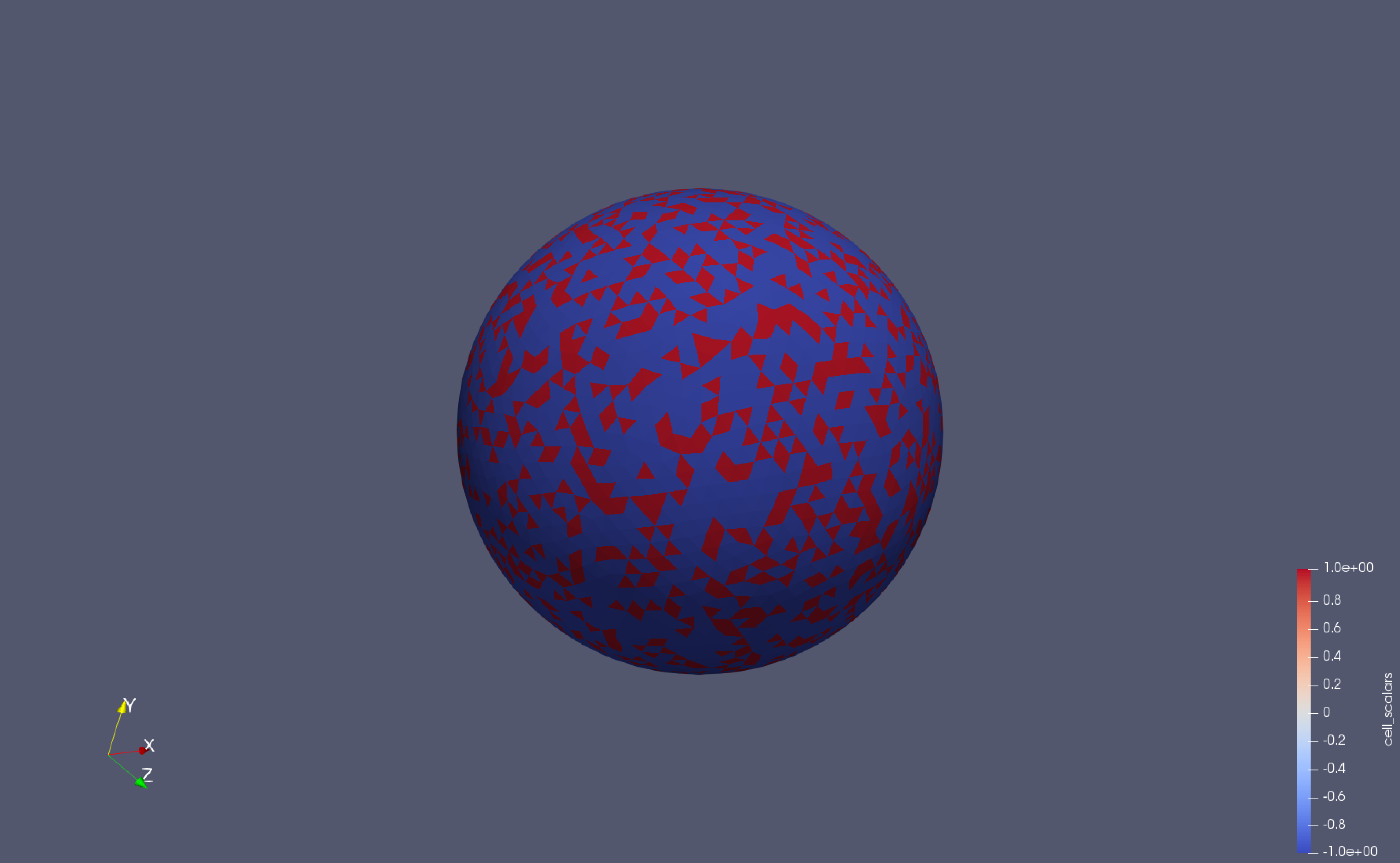,height=1.2in}}
	\subfigure[]{\label{fig:sphere_CH_t_01}\epsfig{figure=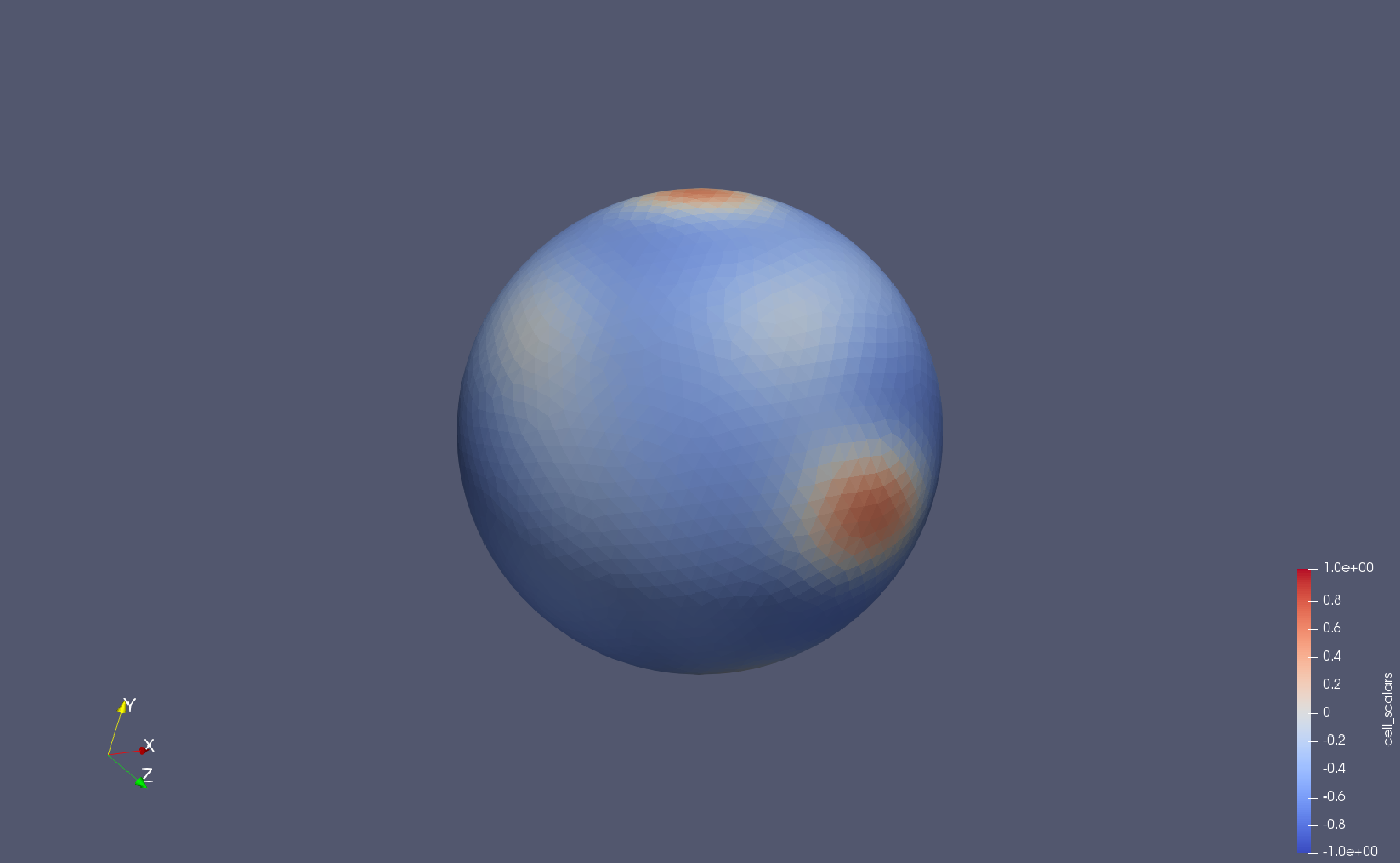,height=1.2in}}
	\subfigure[]{\label{fig:sphere_CH_t_02}\epsfig{figure=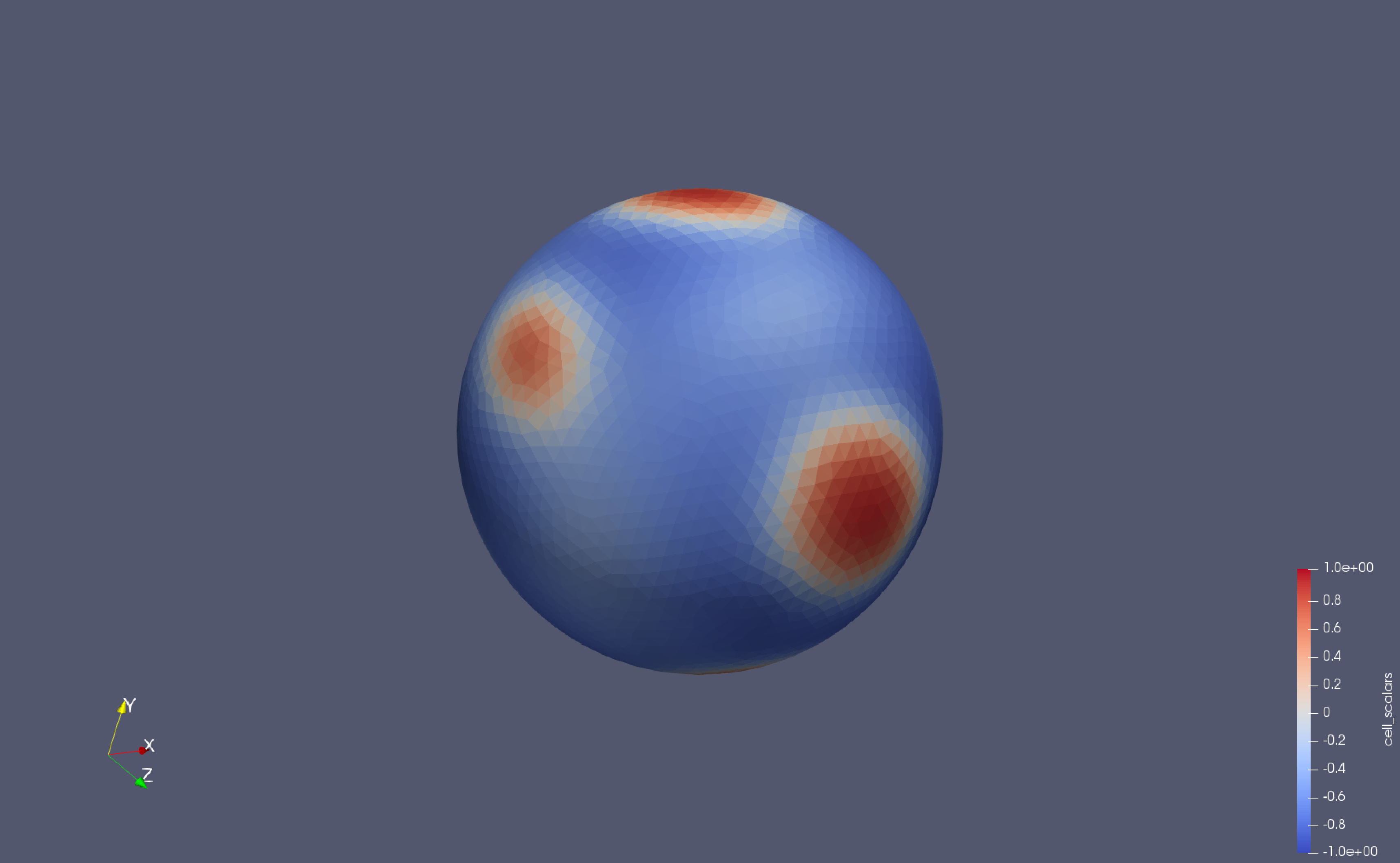,height=1.2in}}
	\subfigure[]{\label{fig:sphere_CH_t_03}\epsfig{figure=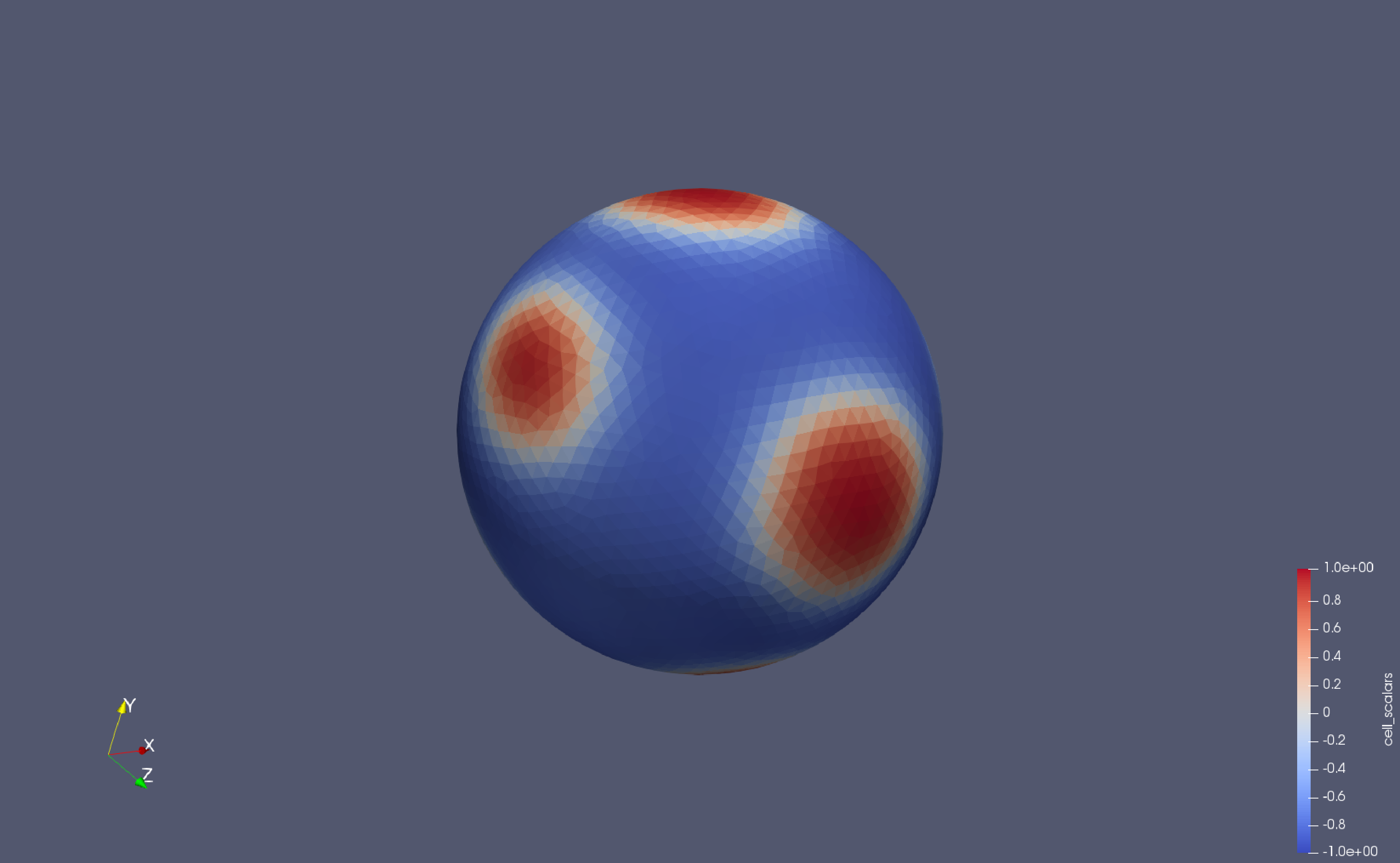,height=1.2in}}
	\caption{Numerical solutions of the concentration $u$ at time $ t = 0.0, 0.1, 0.2, 0.3$ on unit sphere $\mathbb{S}_1$ with 2964 nodes and 5924 triangles.
	}
	\label{fig:sphere_CH}
\end{figure}

\begin{figure}[!ht]
	\centering
	\subfigure[]{\label{fig:ellip_CH_t_0}\epsfig{figure=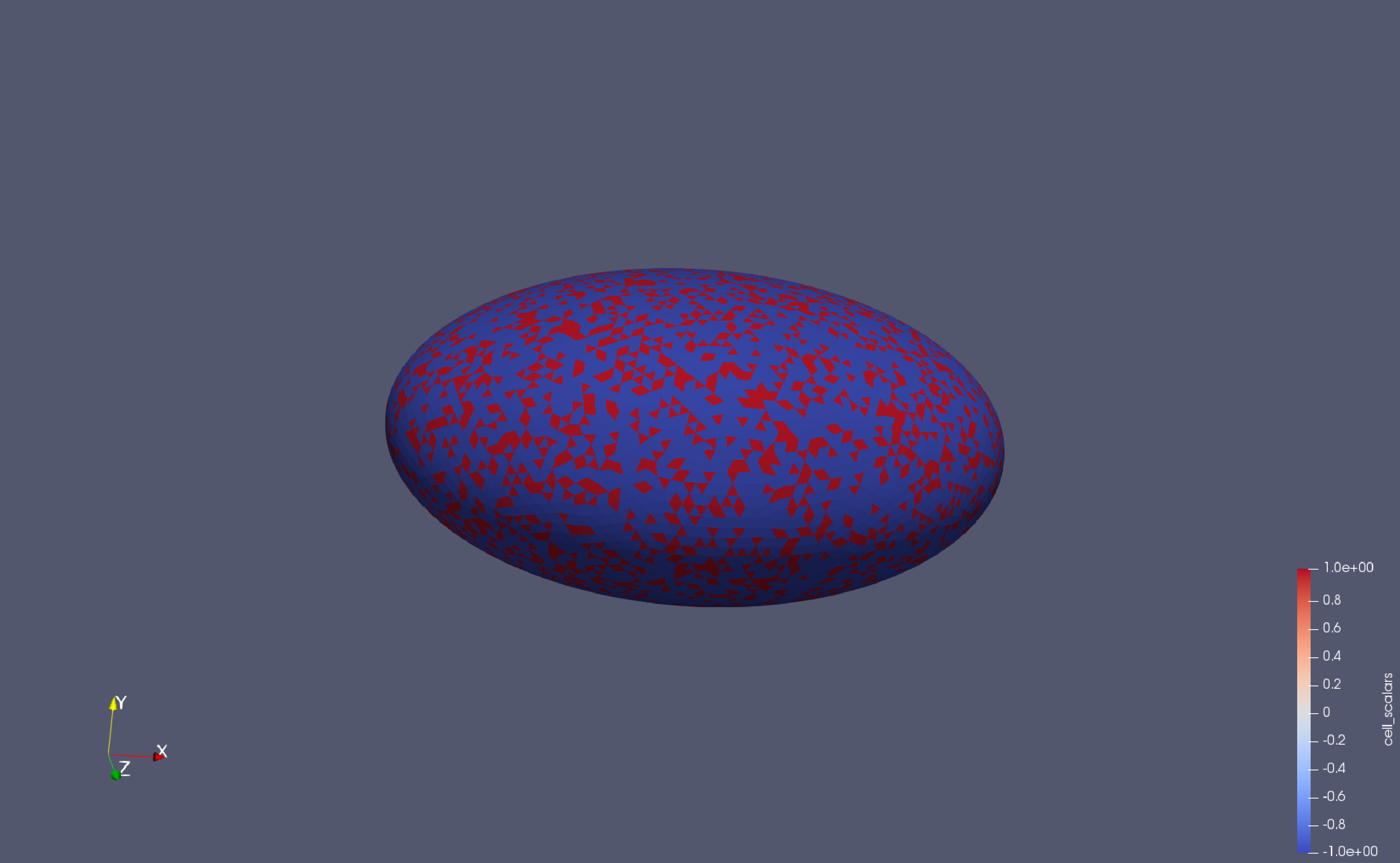,height=1.2in}}
	\subfigure[]{\label{fig:ellip_CH_t_01}\epsfig{figure=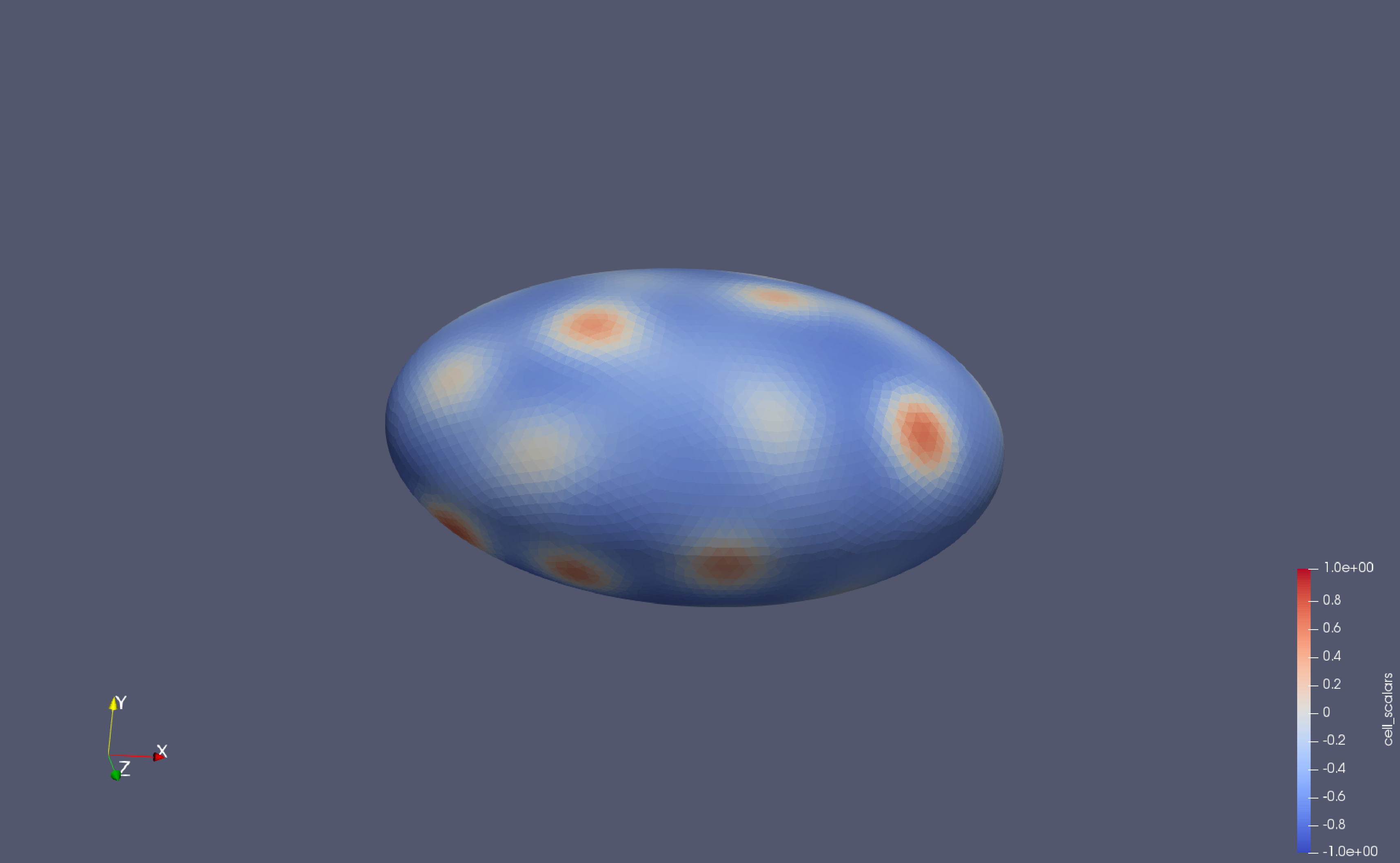,height=1.2in}}
	\subfigure[]{\label{fig:ellip_CH_t_02}\epsfig{figure=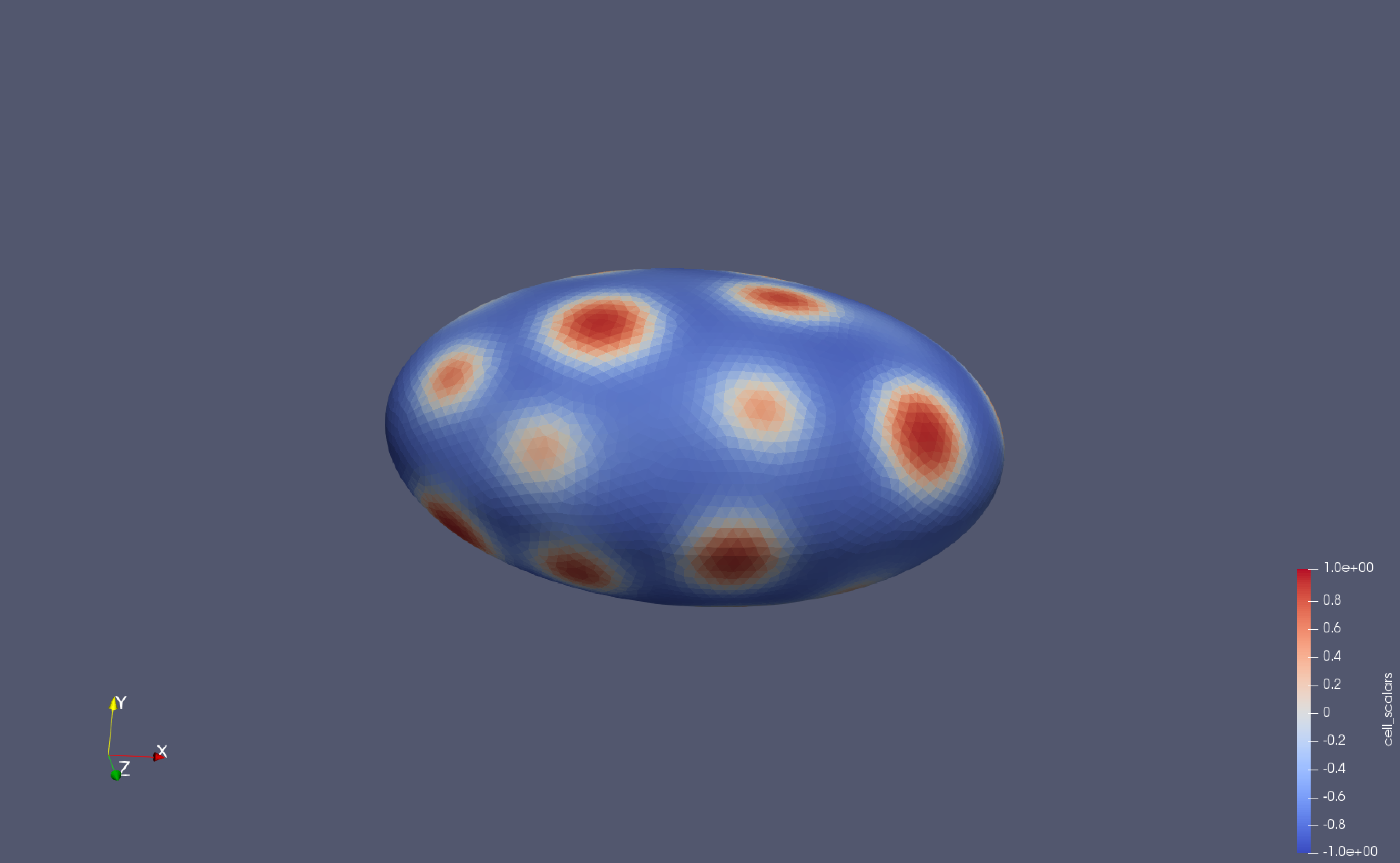,height=1.2in}}
	\caption{Numerical solutions of the concentration $u$ at time $ t = 0.0, 0.1, 0.2$ on ellipsoid $\mathbb{S}_2$ triangulated with 6374 nodes and 12744 triangles.
	}
	\label{fig:ellip_CH}
\end{figure}

\begin{figure}[!ht]
	\centering
	\subfigure[]{\label{fig:rbc_CH_t_0}\epsfig{figure=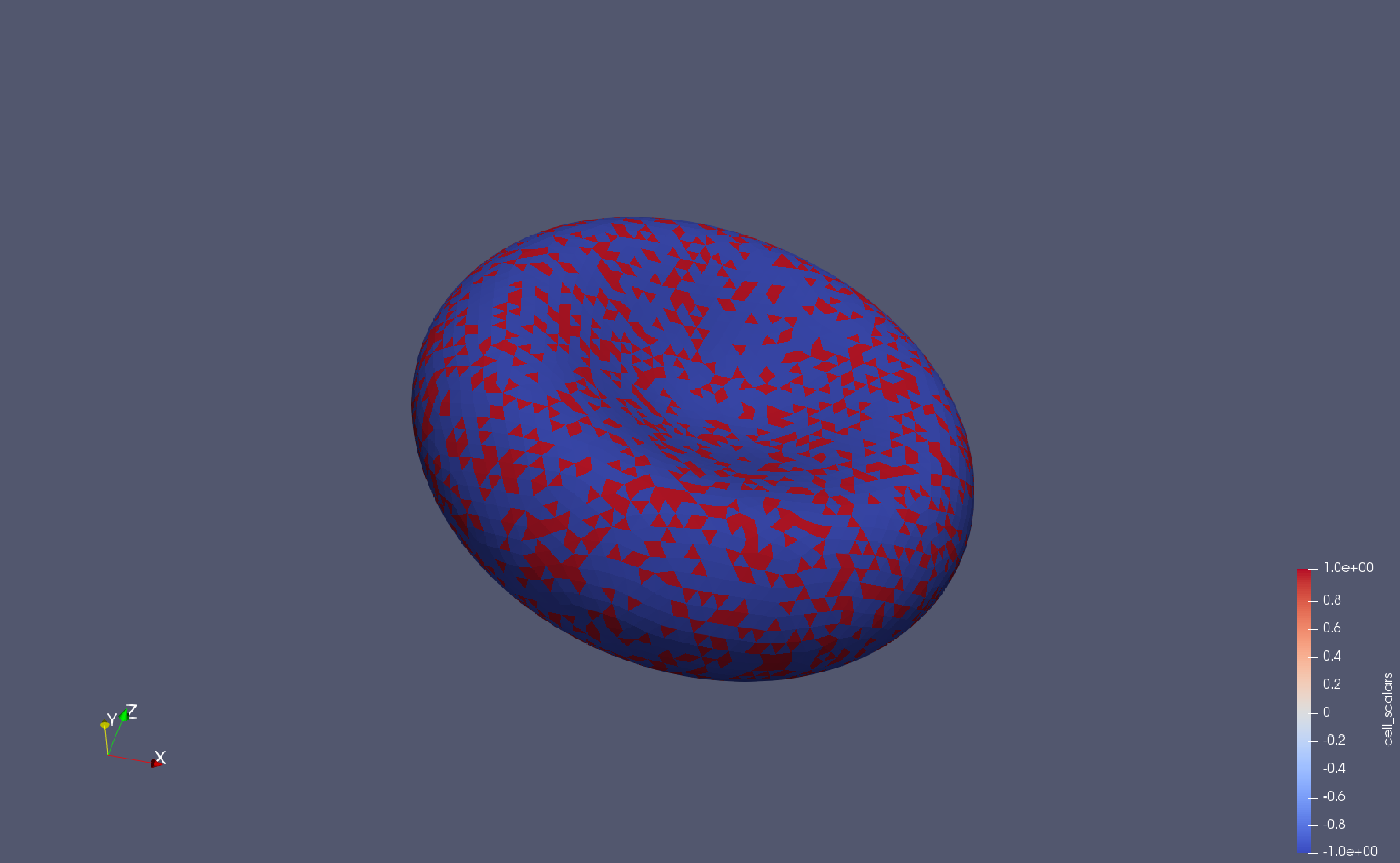,height=1.2in}}
	\subfigure[]{\label{fig:rbc_CH_t_01}\epsfig{figure=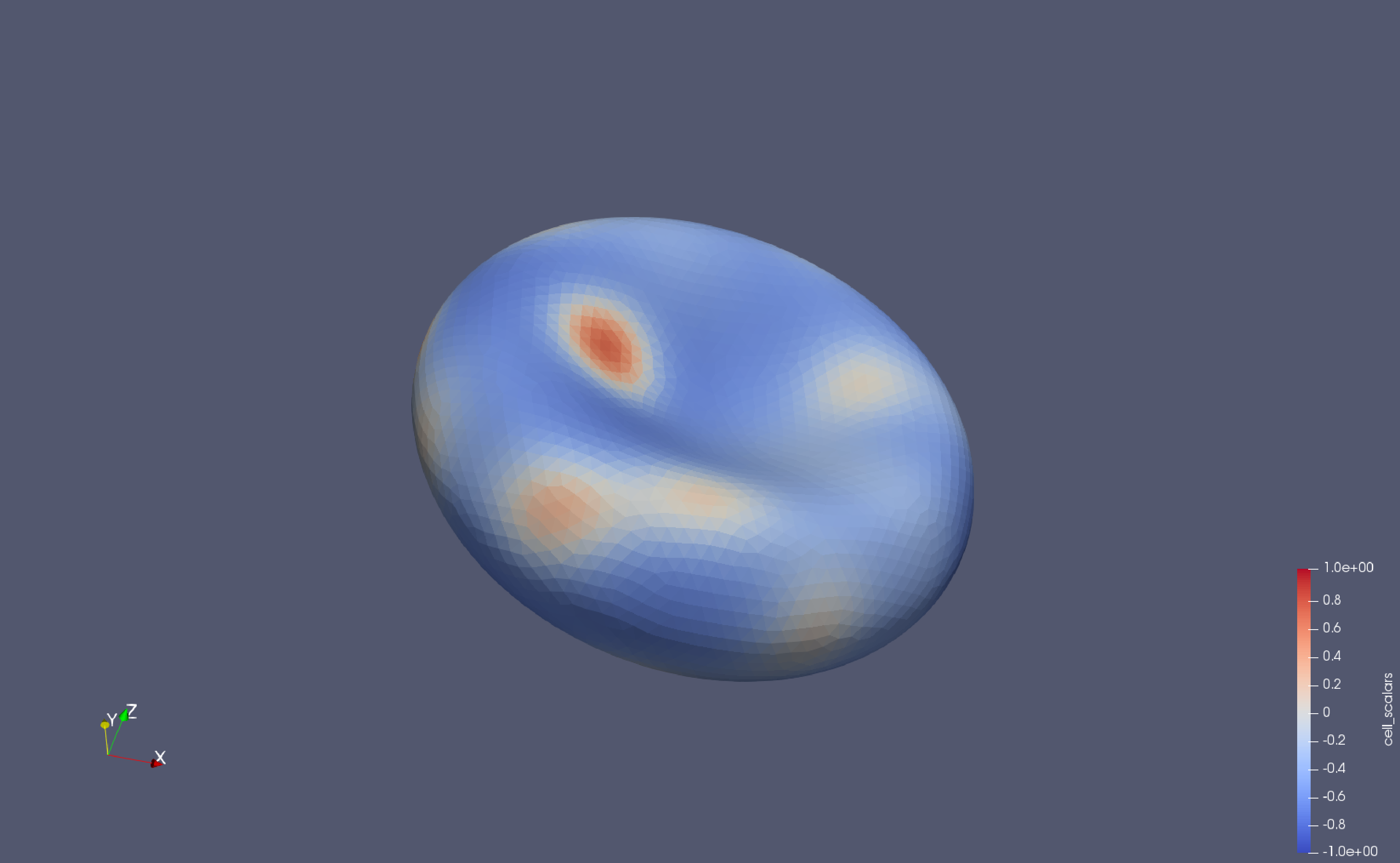,height=1.2in}}
	\subfigure[]{\label{fig:rbc_CH_t_02}\epsfig{figure=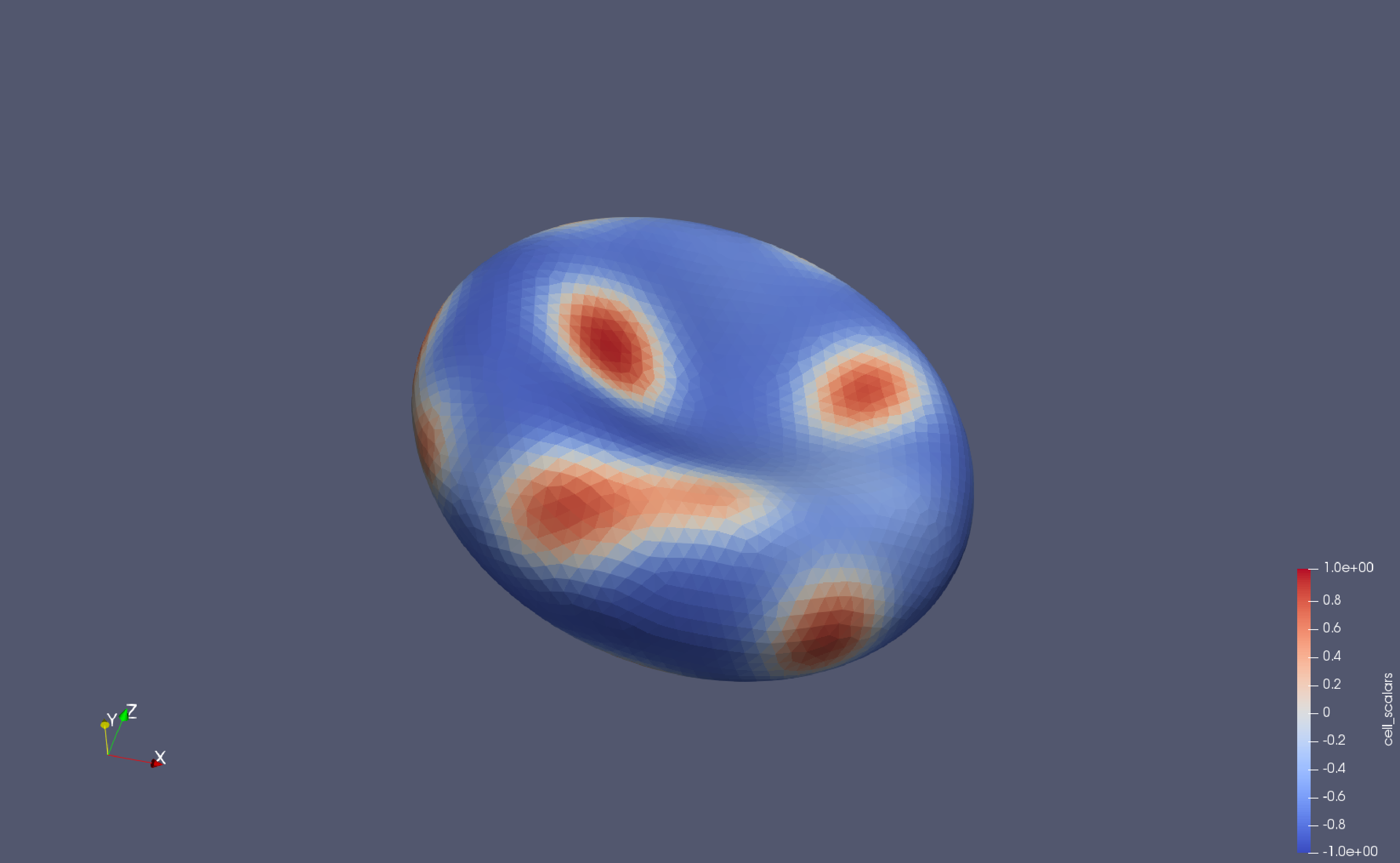,height=1.2in}}
	\subfigure[]{\label{fig:rbc_CH_t_03}\epsfig{figure=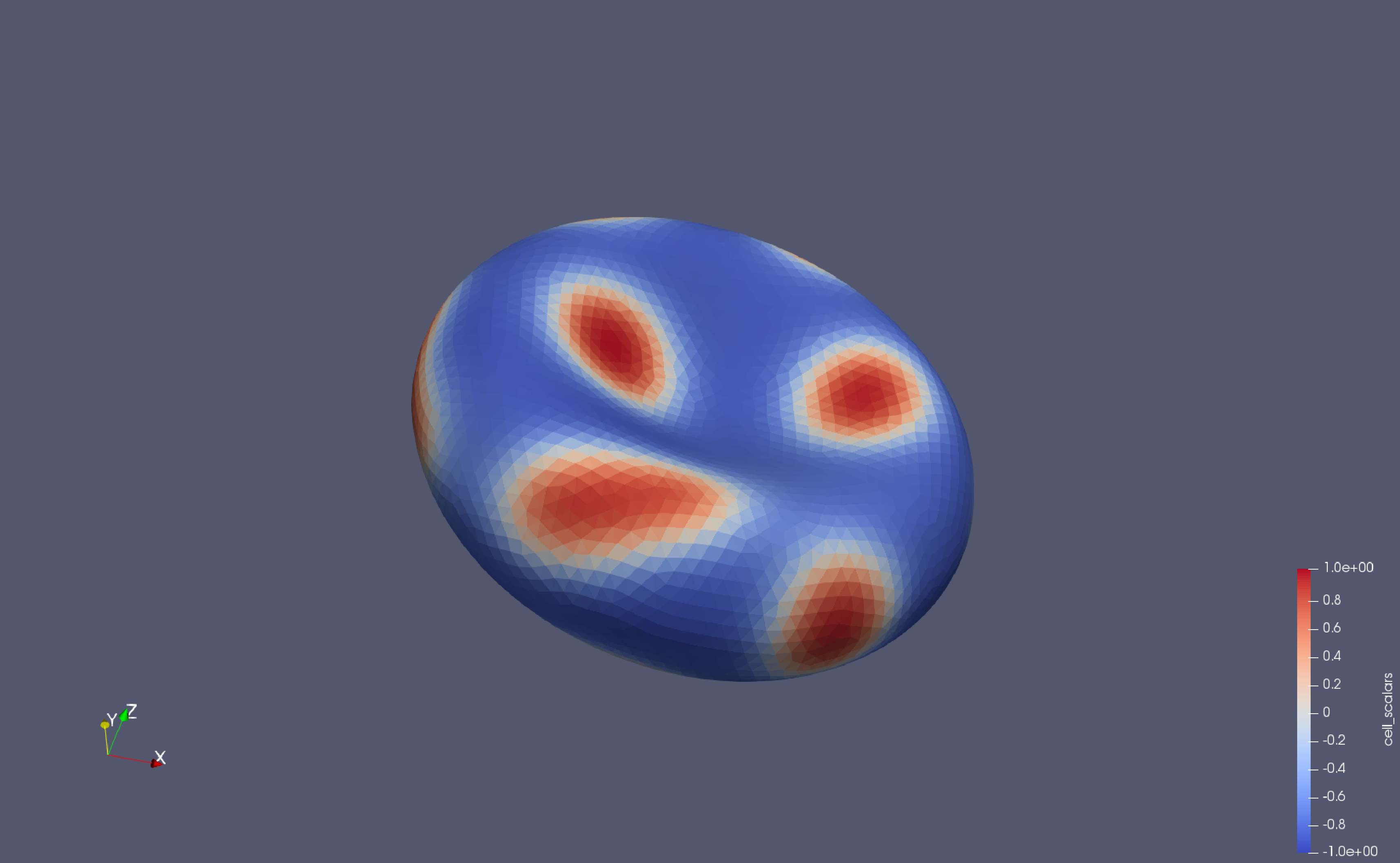,height=1.2in}}
	\caption{Numerical solutions of the concentration $u$ at time $ t = 0.0, 0.1, 0.2, 0.3$ on biconcave surface $\mathbb{S}_3$ with 3904 nodes and 7804 triangles.
	}
	\label{fig:rbc_CH}
\end{figure}

\section{Conclusion}\label{sec:conclusion}
In this paper we introduced LDG schemes for solving initial value problems posed on closed smooth surfaces, which are triangulated by planar triangles. Because of the planar triangle approximation of smooth surfaces, the schemes are second-order accurate.  To further improve order of accuracy of the LDG schemes on surfaces, curved triangles are needed for discretization.
While this paper only implements the second-order accurate schemes, the proposed schemes in principle can be generalized to  meshes consisting of curved elements. However, this would significantly increase programming complexity. We will report in a subsequent paper implementation of these schemes with higher-order accuracy.

Lastly, we used a ${\rm C}_{CFL} \leq 0.1$ in our simulations, which is much smaller than that of LDG for planar domain problems. We will investigate how curvature of the surface affects the CFL condition in the future.

\section{Acknowledgement}

\appendix

\section{Nonlinear Surface Transport }
We consider nonlinear transport equation on a closed manifold
\be\label{eq:nontrans}
\frac{\partial u}{\partial t} +\nabla_{\Gamma}\cdot(\mathbf{f}(u)) =0~.
\ee
The flux $\mathbf{f}$ is some smooth vector field tangent to $\Gamma$. \textbf{We want to show that the DG scheme is stable for solving
Eq.~(\ref{eq:nontrans}) on manifold. The proof follows \cite{CheShu2017,WanShuZha2016}. The main difficulty of the proof is due to the fact that
co-normals of two neighboring elements  are not the same, i.e. $\bn^I \neq -\bn^E$, $\bn^I$ and $\bn^E$ are co-normals defined on the edge shared by the neighboring elements. See also Fig. 1}

The DG method for solving Eqs.~(\ref{eq:nontrans}) is defined by: Find $u_h \in \mathbb{S}_{h,k}$, such that for all test functions $v_h \in \mathbb{S}_{h,k}$,
\be
	\displaystyle{\int_{ \bK_h }} (u_h)_t v_h d\bx -
	\displaystyle{\int_{ \bK_h }} \mathbf{f}(u)\cdot \nabla_{\Gamma_h} v_h d\bx +
	\displaystyle{\int_{ \bK_h }} \left(  \beta_h u_h v_h \right) d\bx
+\displaystyle{\int_{\partial \bK_h}} \reallywidehat{f_h} v_h ds
 = 0 ~,
\label{eq:ldgnt}
\ee
where  $\reallywidehat{f_h}$ is the numerical flux and defined as
\be\label{eq:ntnumflux}
\reallywidehat{f_h}   =  \reallywidehat{f_{K_h}}(u_h^E,u_h^I,\bn^E,\bn^I)= \left\{\mathbf{f}(u_h),\bn_h\right\} + \frac{\alpha}{2} [\![ u_h ]\!]=\frac{1}{2}(\mathbf{f}(u_h^I)\cdot\bn^I-\mathbf{f}(u_h^E)\cdot\bn^E)-\frac{\alpha}{2}(u_h^E-u_h^I),
\ee
with $\alpha = max |\sum_if_i' n_i|$.

\begin{lma}
	If the numerical flux is defined as \eqref{eq:ntnumflux}, it is
	\begin{itemize}
		\item consitstent,
		$ \reallywidehat{f_{K_h}}(u_h,u_h,\bn,\bn) = \mathbf{f}(u_h)\cdot\bn$;
		\item conservative,		 $\reallywidehat{f_{K_h}}(u_h^E,u_h^I,\bn^E,\bn^I) =-\reallywidehat{f_{K_h}}(u_h^I,u_h^E,\bn^I,\bn^E)$;
		\item monotone, non-decrease with $u^I_h$ and non-increase with $u^E_h$ under a suitable $\alpha$.

	\end{itemize}	
\end{lma}

If there exists an entropy flux $F_i(u)$ such that
\be
U'(u)f'_i(u) =F'_i, ~ i=1,2,3,
\ee
then the convex function $U(u)$ is called an enetropy function for Eq. \eqref{eq:nontrans}.
If we denote $v(u) = U'(u)$ and
\be
\psi_i(v) = vf_i(u(v))-F_i(u(v)),
\ee
it is easy to check that
\be
\psi_i'(v) = f_i(u(v)).
\ee

\begin{defn}
A numerical flux $\reallywidehat{f}(u_h^E,u_h^I,\bn^E,\bn^I)$ is called entropy stable with respect to some entropy $U$ if it is consistent, single-valued and satisfies the following inequality
\be
(v^E-v^I)\reallywidehat{f}(u_h^E,u_h^I,\bn^E,\bn^I)\le \sum_i(\psi_i(u_h^E)n_i^E+\psi_i(u_h^I)n_i^I)
\ee
\end{defn}
\begin{thm}
	If $U=\frac {u^2} 2$ is an entropy function of \eqref{eq:nontrans}, and \reallywidehat{f} is entropy stable with respect to $U$, then DG method \eqref{eq:ldgnt} is $L^2$ stable in the sense that
	\be
	\displaystyle{\frac d {dt}\left(\sum_{K_h}\int_{ \bK_h }  \frac {(u_h^I)^2  } 2 d\bx\right)}\le 0
	\ee
\end{thm}
\noindent\textbf{Proof:}
Taking $v_h = u_h$ and summing over all elements yields
\begin{eqnarray}
	&&\displaystyle{\frac d {dt}\left(\sum_{K_h}\int_{ \bK_h }  \frac {(u_h^I)^2  } 2 d\bx\right)} \nonumber\\
	&=&\displaystyle{\sum_{\bK_h}  \int_{ \bK_h }} \mathbf{f}(u)\cdot \nabla_{\Gamma_h} u_h d\bx
-\displaystyle{\sum_{\bK_h}\int_{\partial K_h} \reallywidehat{f_h} u_h^I ds}\nonumber\\
&=& \displaystyle{\sum_{\bK_h}  \int_{ \partial\bK_h }\left(\psi_i(u_h^I)n_i^I  -\reallywidehat{f_h} u_h^I\right) ds}\nonumber\\
&=& \displaystyle{\sum_{e \in \mathcal{E}}}  \int_{e}\left(\psi_i(u_h^I)n_i^I+ \psi_i(u_h^E)n_i^E +\reallywidehat{f_h} (u_h^E-u_h^I)\right) ds \le 0 .
\end{eqnarray}

\textbf{Question: how to design an entropy stable  numerical flux? }
If we use the Lax-Friedrich  flux \eqref{eq:ntnumflux}, I could not prove it is an entropy stable flux with respect to entropy $U = \frac{u^2}2 $ since the conormals of two neighbor elements  are not same.

When $U =\frac{u^2}{2}$, we have  $v=u$,  and $\psi_i =uf_i -F_i$.
\begin{eqnarray}
&&(v_h^E-v_h^I)\reallywidehat{f_{K_h}}(u_h^E,u_h^I,\bn^E,\bn^I)\nonumber\\
&=& (u_h^E-u_h^I)\reallywidehat{f_{K_h}}(u_h^E,u_h^I,\bn^E,\bn^I)\nonumber\\
&=& \frac{1}{2}(u_h^E-u_h^I)(\mathbf{f}(u_h^I)\cdot\bn^I-\mathbf{f}(u_h^E)\cdot\bn^E)-\frac{\alpha}{2}(u_h^E-u_h^I)^2\\
&=& \frac 1 2 (\mathbf{f}(u_h^I)\cdot\bn^Iu_h^E-\mathbf{f}(u_h^I)\cdot\bn^Iu_h^I-\mathbf{f}(u_h^E)\cdot\bn^Eu_h^E+\mathbf{f}(u_h^E)\cdot\bn^Eu_h^I)-\frac{\alpha}{2}(u_h^E-u_h^I)^2\nonumber\\
&=&??
\end{eqnarray}

\end{document}